\documentclass[a4paper]{article}
\usepackage{geometry}
\usepackage{amssymb}
\usepackage{latexsym}
\usepackage{amsmath}
\usepackage{indentfirst}
\usepackage{graphicx}
 \usepackage[colorlinks=true]{hyperref}
\usepackage{mathrsfs} %\mathscr{ABC},...%

\newtheorem{Thm}{Theorem}[section]

\newtheorem{Lem}{Lemma}[section]

\newtheorem{Rek}{Remark}[section]

\newcommand{\R}{\mathbb{R}}

\numberwithin{equation}{section} \numberwithin{figure}{section}

\newenvironment{proof}{\medskip\par\noindent{\bf Proof\/}:}{\qquad
\raisebox{-0.5mm}{\rule{1.5mm}{1mm}}\vspace{6pt}}

\begin{document}
\title{\Large\bf An infinite sequence of localized nodal solutions for Schr\"odinger-Poisson system with double potentials}
\author{
{~~~~Yuanyang Yu$^{a,b}$}\thanks{yuyuanyang18@mails.ucas.ac.cn},  {~~~~Yanheng Ding$^{a,b}$}\thanks{Corresponding author: dingyh@math.ac.cn}\\
\small Institute of Mathematics, Academy of Mathematics and Systems Science\\
\small Chinese Academy of Sciences, Beijing 100190,  P.R.China$^{a}$\\
\small University of Chinese Academy of Sciences, Beijing 100049,  P.R.China$^{b}$\\
}

\date{}

\date{} \maketitle

\noindent{\bf Abstract:}
In this paper, we study the existence of localized sign-changing (or nodal) solutions for the following nonlinear Schr\"odinger-Poisson system
\begin{equation*}
\begin{cases}
-\varepsilon^2 \Delta u+V(x)u+\phi u=K(x)f(u),&\text{in}~\mathbb{R}^3,\\
-\varepsilon^2 \Delta \phi=u^2,&\text{in}~ \mathbb{R}^3,
\end{cases}
\end{equation*}
where $\varepsilon>0$ is small parameters, the linear potential $V$ and nonlinear potential $K$ are bounded and bounded away from zero. By using the penalization method together with the method of invariant sets of descending flow, we establish the existence of an infinite sequence of localized sign-changing solutions which are higher topological type solutions given by the minimax characterization of the symmetric mountain pass theorem and we determine a concrete set as the concentration position of these sign-changing solutions. For single potential, that is, linear potential $V$ or nonlinear potential $K$ is a positive constant, we prove that these localized sign-changing solutions concentrated near a local minimum set of the potential $V$ or a local maximum set of the potential $K$. Moreover, our method is works for the following nonlinear Schr\"odinger equation
\begin{equation*}
-\varepsilon^2 \Delta u+V(x)u=K(x)f(u),~\text{in}~\mathbb{R}^N
\end{equation*}
where $N\geq 2$. The result generalizes the result by Chen and Wang (Calc.Var.Partial Differential Equations 56:1-26, 2017).
\par
\vspace{6mm} \noindent{\bf Keywords:} Schr\"odinger-Poisson system, Localized nodal solutions, Descending flow, Concentration.

\vspace{1mm} \noindent{\bf AMS subject classification:} 35J20, 35J60, 35Q55.

\section{Introduction and main result}
Consider the following nonlinear Schr\"odinger equation
\begin{equation}\label{eq1.1}
i\hbar \partial_t \psi=-\hbar^2\Delta \psi +a(x)\psi +b(x)\phi \psi-f(x,\psi),
\end{equation}
coupled with the Poisson equation
\begin{equation}\label{eq1.2}
-\hbar^2 \Delta \phi=b(x)|\psi|^2,
\end{equation}
where $\hbar>0$ is the Planck constant, the potential $b\in C(\mathbb{R}^3,\mathbb{R}^+)$ and the unknown function $\psi$ is complex that is defined on $\mathbb{R}^3\times [0,\infty)$.  The real functions $a$ and $b$ are defined on $\mathbb{R}^3$ and represent the effective potential and the electric potential, respectively. $f(x,e^{i\theta}\xi)=e^{i\theta}f(x,\xi)$ for $\theta,\xi\in \mathbb{R}$ is a nonlinear function which describes the interaction among many particles. Such problems have been widely investigated due to their deep physics backgrounds. It was introduced in \cite{Benguria-Brezis-Lieb1981CMP,Lieb1981RMP} as a model used in the Thomas-Fermi-von Weizs$\ddot{a}$cker theory in quantum mechanics, it also appeared in semiconductor theory \cite{Benci-Fortunato1998TMNA,Benci-Fortunato2002RMP} to describe solitary waves for nonlinear stationary equations of Schr\"odinger type interacting with an electrostatic field. For more details on the physical aspects of this problem we refer the readers to \cite{Lions1987CMP,Markowich-Ringhofer-Schmeiser1990book,Salvatore2006ANS}.
\par
Standing wave solutions for problem \eqref{eq1.1} and \eqref{eq1.2} have the ansatz form $\psi(x,t)=u(x)e^{\frac{-i\omega t}{\hbar}}$, where $\omega\in \mathbb{R}$ is a constant. And it leads to a system that
\begin{equation}\label{eq1.3}
\begin{cases}
-\varepsilon^2 \Delta u+V(x)u+b(x)\phi u=f(x,u),&\text{in}~ \mathbb{R}^3,\\
-\varepsilon^2 \Delta \phi=b(x)u^2,&\text{in}~ \mathbb{R}^3,
\end{cases}
\end{equation}
with the potential $V(x)=a(x)-\omega$ and $\varepsilon=\hbar$. In the past few decades, the system like or similar to \eqref{eq1.3} has been studied extensively by means of variational tools. See \cite{Azzollini-Avenia-Pomponio2010AIHPNL,Aprile-Mugnai2004PRSE,Ruiz2006JFA,Wang-Zhou2007DCDS,Yang-Ding2010SCM,Zhao-Liu-Zhao2013JDE,Zhao-Zhao2008JMAA} and their references for the existence and multiplicity of solutions. The concentration behavior of solution for system \eqref{eq1.3} has attracted many attentions. By using min-max method and Ljusternik-Schnirelmann theory, He \cite{He2011ZAMP} studied the concentration behavior of positive solutions for system \eqref{eq1.3} with $b(x)\equiv1$ and general nonlinearity $f(u)$ which is subcritical and super 4-linear growth, and obtained the relation between the number of positive solutions and the topology of the global minimum set of $V$ under the assumption that $t^{-3}f(t)$ is increasing on $(0,\infty)$. The critical
case was considered in \cite{He-Zou2012JMP}.  When $b(x)\equiv1$ and $f(x,u)=\lambda|u|^{p-2}u+|u|^4u$ with $3 <p\leq 4$, which it does not satisfy the monotone assumption or Ambrosetti-Rabinowtiz condition, He and Li \cite{He-Li2015AASFM} construct a family of positive solution which concentrates around a local minimum of $V$ as $\varepsilon\rightarrow 0$ under a local condition: there is a bounded domain $\Omega \subset \mathbb{R}^3$ such that
\begin{equation*}
0<\inf_{\Omega}V(x)<\min_{\partial \Omega}V(x).
\end{equation*}
Wang et al. \cite{Wang-Tian-Xu-Zhang2013CVPDE} studied the existence and the concentration behavior of ground state solutions for a subcritical problem with competing potentials. Zhang and Xia \cite{Zhang-Xia2016JDE} considered the critical frequency case, that is $V$ satisfies
\begin{equation*}
0=\inf\limits_{x\in\mathbb{R}^3}V(x)<\liminf\limits_{|x|\rightarrow\infty}V(x)=V_{\infty},
\end{equation*}
and obtained the existence and multiplicity of ground state solutions for system \eqref{eq1.3} with $f(x,u)=|u|^{p-2}u$ and $p\in (4,6)$, which converge to least energy solutions of the associated with limit problem with $V\equiv 0$. Recently, Yang \cite{Yang2020NA} considered the following critical Schr\"odinger-Poisson system
\begin{equation*}
\begin{cases}
-\varepsilon^2 \Delta u+V(x)u+b(x)\phi u=P(x)g(u)+Q(x)|u|^4u,&\text{in}~ \mathbb{R}^3,\\
-\varepsilon^2 \Delta \phi=b(x)u^2,&\text{in}~ \mathbb{R}^3,
\end{cases}
\end{equation*}
and obtained the existence of ground state solutions, which converge to least energy solutions of the associated with limit problem with $b\equiv 0$.
The following Schr\"odinger-Poisson system
\begin{equation}\label{eq1.4}
\begin{cases}
-\varepsilon^2 \Delta u+V(x)u+b(x)\phi u=u^p,&\text{in}~\mathbb{R}^3,\\
-\Delta \phi=b(x)u^2,&\text{in}~\mathbb{R}^3,
\end{cases}
\end{equation}
has also attracted many scholars' attention. Ianni and Vaira \cite{Ianni-Vaira2008ANS} obtained the existence of positive bound state solutions of system \eqref{eq1.4} with $p\in (1,5)$, which concentrate at a non-degenerate local minimum or maximum of $V$ by using a Lyapunov-Schmidt reduction method.
Ruiz and Vaira \cite{Ruiz-Vaira2011RMI} constructed multi-bump solutions of system \eqref{eq1.4} with $b(x)=1$ and $p\in (1,5)$ and these bumps concentrate around a local minimum of the potential $V$. D'Aprile and Wei \cite{Aprile-Wei2005SIAM} showed that system \eqref{eq1.4} with $V(x)=b(x)=1$ possesses a family of radially symmetric solutions concentrating around a sphere as $\varepsilon\rightarrow 0$ for $p\in (1,\frac{11}{7})$.
\par
Another topic which has received increasingly interest of late years is the existence of sign-changing solutions of system \eqref{eq1.3} with $\varepsilon=1$. Using a Nehari-type manifold and gluing solution pieces, Kim and Seok \cite{Kim-Seok2012CCM} proved the existence of radial sign-changing solutions with prescribed numbers of nodal domains for \eqref{eq1.3} in the case where $V(x)=b(x)=1,f(x,u)=|u|^{p-2}u$ with $p\in (4,6)$. For same nonlinearity, using the constraint variational method and the Brouwer degree theory, Wang and Zhou \cite{Wang-Zhou2015CVPDE} proved that system \eqref{eq1.3} has a sign-changing solution under suitable assumptions. Using the method of invariant sets of descending flow, Liu et al. \cite{Liu-Wang-Zhang2016AdM} obtained the existence of infinitely many sign-changing solutions for $3$-superlinear nonlinearity and a coercive potential
 function. Huang et al. \cite{Huang-Rocha-Chen2013JMAA} consider the following critical Schr\"odinger-Poisson system
\begin{equation}\label{eq1.5}
\begin{cases}
-\Delta u+u+b(x)\phi u=k(x)|u|^4u+\mu h(x)u,&\text{in}~ \mathbb{R}^3,\\
-\Delta \phi=b(x)u^2,&\text{in}~\mathbb{R}^3,
\end{cases}
\end{equation}
and proved the existence of at least a pair of fixed sign solutions and a pair of sign-changing solutions for system \eqref{eq1.5} under some suitable conditions on the nonnegative functions $b, k$ and $h$. Later, Zhong and Tang \cite{Zhong-Tang2018NAWRA} also obtained system \eqref{eq1.5} with $k(x)\equiv1$ possesses at
least one ground state sign-changing solution by constraint variational method and its energy is strictly larger than twice that of ground state solutions.
\par
\vspace{3mm}
In quantum physics, the parameter $\varepsilon$ is generically quite small. In general, one expects to recover some classical dynamics in the semi-classical limit r$\acute{e}$gime. The semi-classical limit is well understood for the nonlinear Schr\"odinger equation, i.e., equation \eqref{eq1.1} with $b(x)\equiv0$,
\begin{equation}\label{eq1.6}
-\varepsilon^2\Delta u+V(x)u=f(u),~\text{in}~\mathbb{R}^N.
\end{equation}
The study of equation \eqref{eq1.6} goes back to the pioneer works by Floer-Weinstein \cite{Floer-Weinstein1986JFA} and Rabinowitz \cite{Rabinowitz1992ZAMP}. And since then it has been studied extensively under various assumptions on the potential and the nonlinearity, see for example \cite{Ambrosetti-Badiale-Cingolani1997ARMA,Ambrosetti-Malchiodi-Secchi2001ARMA,Bartsch-Wang1995CPDE,Byeon-Wang2002ARMA,Byeon-Wang2003CVPDE,DelPino-Felmer1996CVPDE,DelPino-Felmer1998AIHPNL,Wang1993CMP} and references therein for concentration behavior od solution. In particular, when $V$ has a local minimum point P, ground state and bound state positive solutions with a single spike concentrating near a local minimum point can be constructed. These solutions are mountain pass type critical points having Morse index 1 and they are positive solutions obtained as perturbations of mountain pass positive solutions of the corresponding limiting equation
\begin{equation}\label{eq1.7}
-\Delta u+V(P)u=f(u),~\text{in}~\mathbb{R}^N.
\end{equation}
 Positive solutions with multi-peaks concentrating near a maximum point or a saddle point of $V$ are also constructed by using the Lyapunov-Schmidt reduction method (e.g., \cite{Kang-Wei2000ADE}). These solutions have multiple peaks clustered at the given maximum point of $V$ or a saddle point of $V$ and also called cluster solutions. Furthermore, Kang and Wei \cite{Kang-Wei2000ADE} showed surprisingly that there can not be multi-peak positive solutions concentrating near a local minimum point. Therefore, in order to have other localized solutions near a local minimum point one has to look for sign-changing solutions with concentrations. Progress has been made in recent years in this direction, see \cite{Alves-Soares2004JMAA,Bartsch-Clapp-Weth2007MA,Aprile-Pistoia2007MA,Aprile-Pistoia2009AIHPNL,Aprile-Ruiz2011MZ} and references therein, in which a finite number of localized sign-changing solutions can be constructed by various different methods. All the known localized sign-changing solutions largely depends on the Lyapunov-Schmidt reduction method or the non-degeneracy condition of the mountain pass solutions of the limit equation \eqref{eq1.7}. However, Chen and Wang \cite{Chen-Wang2017CVPDE} proved that there exists an infinite sequence of sign-changing solutions which concentrate at any given strict local minimum point of $V$ as $\varepsilon\rightarrow 0$ without using any non-degeneracy conditions and there solutions are higher topological type (in the sense that critical points obtained by the symmetric mountain pass theorem using higher dimensional symmetric linking structures) solutions given by the minimax characterization of the symmetric mountain pass theorem. The critical case was considered in \cite{Chen-Liu-Wang2019JFA}. Their method is quite effective and can be generalized to a class of quasilinear Schr\"odinger equation with subcritical nonlinearity growth \cite{Liu-Liu-Wang2019JDE} and p-Laplacian equations with critical exponents \cite{Gao-Guo2020JMP}.
\par
\vspace{3mm}
Motivated by the works described above, the aim of this paper is to continue to study the existence and
concentration behavior of sign-changing solutions for Schr\"odinger-Poisson system, that is, we consider the following nonlinear Schr\"odinger-Poisson system
\begin{equation}\label{eq1.8}
\begin{cases}
-\varepsilon^2\Delta u+V(x)u+\phi u=K(x)f(u),&\text{in}\ \mathbb{R}^3,\\
-\varepsilon^2\Delta \phi=u^2,&\text{in}\ \mathbb{R}^3,
\end{cases}
\end{equation}
where $\varepsilon>0$ is a small parameter, the linear potential $V$ and nonlinear potential $K$ are bounded and
bounded away from zero.
\par
An interesting question, which motivates the present work, is whether one can find multiple sign-changing solutions of nonlinear Schr\"odinger-Poisson system with double potentials, and these sign-changing solutions concentrate around critical points of the potential, and whether the sign-changing solutions of higher topological type can be localized. As far as we know, such a problem was not considered before. There are some difficulties in such a problem. The first one is that such problem involves two different potentials which make our problem more complicated than that of \cite{Chen-Wang2017CVPDE}, and we need to find a suitable set as the concentration position. The second one is, as we will see, the term $\phi$ is nonlocal, and this need more detailed estimates on the local Pohozaev identity (see Lemma \ref{Lem4.7}) in order to show that the sign-changing solutions concentrate near a given concrete set.
\par
In this paper, we will given an answer to the above question. First, we introduce a modified problem by the penalization method introduction by del Pino and Felmer \cite{DelPino-Felmer1996CVPDE} and we obtain the existence of an infinite sequence of localized sign-changing solutions of modified problem by using the method of invariant sets of descending flow for $\varepsilon>0$ small enough. To study the concentration behavior of these solutions, we establish the $L^\infty$ and decay estimate of these solutions. At last, we determine a concrete set as the concentration position of these solutions and prove these localized sign-changing solutions are indeed solutions of original problem.

\par
To state our main results, we need the following assumptions
\begin{itemize}
\item[$(f_1)$]  $f\in C^1(\mathbb{R},\mathbb{R}),f(t)=o(t)$ as $|t|\rightarrow 0$.
\item[$(f_2)$]  there are constants $c_1,c_2>0$ and $4<p<2^*(=6)$ such that for all $t\in \mathbb{R}$
\begin{equation*}
|f(t)|\leq c_1+c_2|t|^{p-1}.
\end{equation*}
\item[$(f_3)$]  there is a constant $\mu>4$ such that for any $t\in \mathbb{R}\setminus \{0\}$
\begin{equation*}
0<\mu F(t)\leq f(t)t,~\text{where}~F(t)=\int_0^tf(\tau)d\tau.
\end{equation*}
\item[$(f_4)$]  $f$ is odd in $t$, i.e., $f(-t)=-f(t)$ for all $t\in \mathbb{R}$.
\end{itemize}
\begin{Rek}\label{Rek1.1}
Note that, $(f_1)$ and $(f_2)$ imply that for each $\delta>0$, there is $C(\delta)>0$ such that
\begin{equation}\label{eq1.9}
f(t)\leq \delta t+C(\delta) t^{p-1}~~~\text{and}~~~F(t)\leq\delta t^2+C(\delta)t^p,~~\forall t\geq 0.
\end{equation}
By $(f_3)$, we deduce that there exist $C_0,C_1>0$ such that
\begin{equation}\label{eq1.10}
F(t)\geq C_0t^\mu-C_1t^2,~\forall t\geq 0.
\end{equation}
\end{Rek}
We also assume the potentials $V$ and $K$ satisfy that
\begin{itemize}
\item[$(V)$]  $V\in C^1(\mathbb{R}^3,\mathbb{R})$ and there are two constants $a_1,a_2$ with $0<a_1<a_2$ such that
\begin{equation*}
a_1\leq V(x)\leq a_2,~~~\forall x\in \mathbb{R}^3.
\end{equation*}
\item[$(K)$]  $K\in C^1(\mathbb{R}^3,\mathbb{R})$ and there are two constants $b_1,b_2$ with $0<b_1<b_2$ such that
\begin{equation*}
b_1\leq K(x)\leq b_2,~~~\forall x\in \mathbb{R}^3.
\end{equation*}
\end{itemize}
In the following, we propose two kinds of assumptions. In the first one, we assume
\begin{itemize}
\item[$(VK1)$]  There exists a bounded domain $\Lambda\subset \mathbb{R}^3$ with smooth boundary $\partial \Lambda$ such that
\begin{equation}\label{eq1.11}
\vec{n}(x)\cdot \nabla V(x)>0,~~~\forall x\in \partial \Lambda,
\end{equation}
\begin{equation}\label{eq1.12}
\nabla K(x)\cdot \nabla V(x)<0,~~~\forall x\in \partial \Lambda,
\end{equation}
and where $\vec{n}(x)$ denotes the outward normal vector of $\partial \Lambda$ at $x$.
\end{itemize}
Note that \eqref{eq1.11} is satisfied if $V$ has an isolated local minimum set, i.e., $V$ has a local trapping potential well. Under the assumption \eqref{eq1.11}, the set of critical points of $V$ inside $\Lambda$
\begin{equation}\label{eq1.13}
\mathcal{A}=\{x\in \Lambda|\nabla V(x)=0\},
\end{equation}
is a nonempty and $\mathcal{A}$ is a compact subset of $\Lambda$. Without loss of generality, we may assume $0\in \mathcal{A}$. In order to describe concentration phenomena of localized sign-changing solutions, we define the concentration set by $U(\delta)=\{x\in \Lambda|\text{dist}(x,\partial \Lambda)>\delta\}$ for $\delta>0$ small.
\par
To state our results we first set some notations. For any subset $B\subset\mathbb{R}^3$ and $\delta> 0$, we define
\begin{equation*}
B^\delta=\{x\in \mathbb{R}^3|\text{dist}(x,B):=\inf_{y\in B}|x-y|<\delta\}.
\end{equation*}
A function $u\in H^1(\mathbb{R}^3)$ is called sign-changing if $u^+\neq 0$ and $u^-\neq 0$, where
\begin{equation*}
u^+(x)=\max\{u(x),0\}~~\text{and}~~u^-(x)=\min\{u(x),0\}.
\end{equation*}
\par
The first result of this paper is stated as follows.
\begin{Thm}\label{Thm1.1}
Assume that $(f_1)$-$(f_4), (V), (K)$ and $(VK1)$ hold. Then for any positive integer $k$, there exists $\varepsilon_k > 0$ such that if $0 <\varepsilon <\varepsilon_k$, system \eqref{eq1.8} has at least $k$ pairs of sign-changing solutions $\pm v_{j,\varepsilon}, j = 1, 2, \cdots, k$. Moreover, there exists $\delta_0>0$, for any $\delta\in (0,\delta_0)$, there exist $c = c(\delta, k) > 0$ and $C = C(\delta, k) > 0$ such that
\begin{equation*}
|v_{j,\varepsilon}(x)|\leq C\text{exp}\bigg(-\frac{c\text{dist}(x,U(\delta)}{\varepsilon}\bigg)~\text{for}~x\in \mathbb{R}^3,~j=1,\cdots,k.
\end{equation*}
\end{Thm}
\par
\vspace{3mm}
Secondly, we set the dual case of the first one by supposing that
\begin{itemize}
\item[$(VK2)$]  There exists a bounded domain $\Lambda\subset \mathbb{R}^3$ with smooth boundary $\partial \Lambda$ such that
\begin{equation}\label{eq1.14}
\vec{n}(x)\cdot \nabla K(x)<0,~~~\forall x\in \partial \Lambda,
\end{equation}
\begin{equation}\label{eq1.15}
\nabla K(x)\cdot \nabla V(x)<0,~~~\forall x\in \partial \Lambda.
\end{equation}
\end{itemize}
\par
Similarly, \eqref{eq1.14} is satisfied if $K$ has an isolated local maximum set. Under the assumption \eqref{eq1.14}, the set of critical points of $K$ inside $\Lambda$
\begin{equation}\label{eq1.16}
\mathcal{B}=\{x\in \Lambda|\nabla K(x)=0\},
\end{equation}
is nonempty and $\mathcal{B}$ is a compact subset of $\Lambda$. In this case, we also assume $0\in \mathcal{B}$.
\par
We have a dual result of Theorem \ref{Thm1.1}.
\begin{Thm}\label{Thm1.2}
Assume that $(f_1)$-$(f_4), (V), (K)$ and $(VK2)$ hold. Then for any positive integer $k$, there exists $\varepsilon_k > 0$ such that if $0 <\varepsilon <\varepsilon_k$, system \eqref{eq1.8} has at least $k$ pairs of sign-changing solutions $\pm v_{j,\varepsilon}, j = 1, 2, \cdots, k$. Moreover, there exists $\delta_0>0$, for any $\delta\in (0,\delta_0)$, there exist $c = c(\delta, k) > 0$ and $C = C(\delta, k) > 0$ such that
\begin{equation*}
|v_{j,\varepsilon}(x)|\leq C\text{exp}\bigg(-\frac{c\text{dist}(x,U(\delta)}{\varepsilon}\bigg)~\text{for}~x\in \mathbb{R}^3,~j=1,\cdots,k.
\end{equation*}
\end{Thm}

Next we consider single potential, that is, we consider the following nonlinear Schr\"odinger-Poisson system
\begin{equation}\label{eq1.17}
\begin{cases}
-\varepsilon^2 \Delta u+V(x)u+\phi u=K_0f(u),&\text{in}~\mathbb{R}^3,\\
-\varepsilon^2 \Delta \phi=u^2,&\text{in}~\mathbb{R}^3,
\end{cases}
\end{equation}
and
\begin{equation}\label{eq1.18}
\begin{cases}
-\varepsilon^2 \Delta u+V_0u+\phi u=K(x)f(u),&\text{in}~\mathbb{R}^3,\\
-\varepsilon^2 \Delta \phi=u^2,&\text{in}~\mathbb{R}^3,
\end{cases}
\end{equation}
where $V_0$ and $K_0$ are two positive constants.
\begin{Thm}\label{Thm1.3}
Assume that $(f_1)$-$(f_4), (V)$ and \eqref{eq1.11} in $(VK1)$ hold. Then for any positive integer $k$, there exists $\varepsilon_k > 0$ such that if $0 <\varepsilon <\varepsilon_k$, system \eqref{eq1.17} has at least $k$ pairs of sign-changing solutions $\pm v_{j,\varepsilon}, j = 1, 2, \cdots, k$. Moreover, for any $\delta> 0$, there exist $c = c(\delta, k) > 0$ and $C = C(\delta, k) > 0$ such that
\begin{equation*}
|v_{j,\varepsilon}(x)|\leq C\text{exp}\bigg(-\frac{c\text{dist}(x,\mathcal{A}^\delta)}{\varepsilon}\bigg)~\text{for}~x\in \mathbb{R}^3,~j=1,\cdots,k.
\end{equation*}
\end{Thm}
We also have a dual result of Theorem \ref{Thm1.3}.

\begin{Thm}\label{Thm1.4}
Assume that $(f_1)$-$(f_4), (K)$ and \eqref{eq1.14} in $(VK2)$ hold. Then for any positive integer $k$, there exists $\varepsilon_k > 0$ such that if $0 <\varepsilon <\varepsilon_k$, system \eqref{eq1.18} has at least $k$ pairs of sign-changing solutions $\pm v_{j,\varepsilon}, j = 1, 2, \cdots, k$. Moreover, for any $\delta> 0$, there exist $c = c(\delta, k) > 0$ and $C = C(\delta, k) > 0$ such that
\begin{equation*}
|v_{j,\varepsilon}(x)|\leq C\text{exp}\bigg(-\frac{c\text{dist}(x,\mathcal{B}^\delta)}{\varepsilon}\bigg)~\text{for}~x\in \mathbb{R}^3,~j=1,\cdots,k.
\end{equation*}
\end{Thm}
\par
\vspace{3mm}
In the sequel, we only give the detailed proof for Theorem \ref{Thm1.1} and Theorem \ref{Thm1.3} because the argument for Theorem \ref{Thm1.2} and Theorem \ref{Thm1.4} is similar to that for Theorem \ref{Thm1.1} and Theorem \ref{Thm1.3}, respectively.
\par
An outline of this paper is as follows. In section 2, we first present the variational
setting of the problem, and we modify the original problem. Then we prove the $(PS)$ condition for the modified functional. In section 3, we prove the existence of multiple sign-changing solutions of modified problem by adapting an abstract critical point theorem in \cite{Liu-Liu-Wang2016JDE}. In section 4, we give the proof of Theorem \ref{Thm1.1} and Theorem \ref{Thm1.3}. In the last section, we make a remark about a related problem.
\par
\vspace{3mm}
{\bf Notation.~}Throughout this paper, we make use of the following notations.
\begin{itemize}
\item[$\bullet$]  For any $R>0$ and for any $x\in \mathbb{R}^3$, $B_R(x)$ denotes the ball of radius $R$ centered at $x$.
\item[$\bullet$]  $L^q(\mathbb{R}^3),1\leq q<\infty$ denotes the usual Lebesgue space with norm
    \begin{equation*}
    \|u\|_p=\big(\int_{\mathbb{R}^3}|u|^pdx\big)^{\frac{1}{p}}.
    \end{equation*}
\item[$\bullet$]  $H^1(\mathbb{R}^3)$ and $D^{1,2}(\mathbb{R}^3)=\{u\in L^6(\mathbb{R}^3)||\nabla u|\in L^2(\mathbb{R}^3)\}$ denote the usual Sobolev space, endowed with norm
\begin{equation*}
\|u\|=\big(\int_{\mathbb{R}^3}[|\nabla u|^2+u^2]dx\big)^{\frac{1}{2}}~~~resp.~~~\|u\|_D=\big(\int_{\mathbb{R}^3}|\nabla u|^2dx\big)^{\frac{1}{2}}.
\end{equation*}
\item[$\bullet$]  $o_n(1)$ denotes $o_n(1)\rightarrow 0$ as $n\rightarrow\infty$.
\item[$\bullet$]  $C$ or $C_i(i=1,2,\cdots)$ are some positive constants may change from line to line.
\end{itemize}

\section{Variational setting and preliminary results}
\subsection{Variational setting}
Recall that by the Lax-Milgram theorem, we know that for every $u\in H^1(\mathbb{R}^3)$, there exists a unique $\phi_u\in D^{1,2}(\mathbb{R}^3)$ such that
\begin{equation*}
-\Delta\phi_u=u^2
\end{equation*}
 and $\phi_u$ can be expressed by
\begin{equation*}
\phi_u(x)=\frac{1}{4\pi}\frac{1}{|x|}*u^2=\frac{1}{4\pi}\int_{\mathbb{R}^3}\frac{u^2(y)}{|x-y|}dy,~\forall\,x\in \mathbb{R}^3,
\end{equation*}
which is called Riesz potential (see \cite{Lieb-Loss2001book}), and $*$ denotes the convolution operator. It is clear that $\phi_u(x)\geq 0$ for all $x\in \mathbb{R}^3$. We will omit the constant $\pi$ in the sequel.
\par
Making the change of variable $x\mapsto \varepsilon x$, we can rewrite the system \eqref{eq1.8} as the following equivalent form
\begin{equation}\label{eq2.1}
\begin{cases}
-\Delta u+V(\varepsilon x)u+\phi u=K(\varepsilon x)f(u),~&\text{in}~\mathbb{R}^3,\\
-\Delta \phi =u^2,~&\text{in}~\mathbb{R}^3.
\end{cases}
\end{equation}
If $u$ is a solution of the system \eqref{eq2.1}, then $v(x):=u(\frac{x}{\varepsilon})$ is a solution of the system \eqref{eq1.8}. Thus, to study the system \eqref{eq1.8}, it suffices to study the system \eqref{eq2.1}. Then the system \eqref{eq2.1} can be reduced to the Schr\"{o}dinger equation with nonlocal term:
\begin{equation}\label{eq2.2}
-\Delta u+V(\varepsilon x)u+\phi_uu=K(\varepsilon x)f(u),~~\text{in}\ \mathbb{R}^3.
\end{equation}
Moreover, it can be proved that $u\in H^1(\mathbb{R}^3)$ is a solution of system \eqref{eq2.1} if and only if $u\in H^1(\mathbb{R}^3)$ is a critical point of the functional $I_\varepsilon:H^1(\mathbb{R}^3)\rightarrow \mathbb{R}$ defined as
\begin{equation*}
I_\varepsilon(u)=\frac{1}{2}\int_{\mathbb{R}^3}|\nabla u|^2dx+\frac{1}{2}\int_{\mathbb{R}^3}V(\varepsilon x)u^2dx+\frac{1}{4}\int_{\mathbb{R}^3}\phi_u u^2dx-\int_{\mathbb{R}^3}K(\varepsilon x)F(u)dx,
\end{equation*}
where $\phi_u$ is the unique solution of the second equation in \eqref{eq2.1}.
\par
Because we are concerned with the non-local problem in view of the presence of term $\phi_u$, we would like to recall the well-known Hardy-Littlewood-Sobolev inequality.
\begin{Lem}[Hardy-Littlewood-Sobolev inequality, \cite{Lieb-Loss2001book}]\label{Lem2.1}
Let $t,r>1$ and $0<\mu<3$ with
\begin{equation*}
\frac{1}{t}+\frac{\mu}{3}+\frac{1}{r}=2,
\end{equation*}
$g\in L^t(\R^3)$ and $h\in L^r(\R^3)$. There exists a sharp constant $C(t,\mu,r)$, independent of $f,h$ such that
\begin{equation*}
\int_{\mathbb{R}^3}\int_{\mathbb{R}^3}\frac{g(x)h(y)}{|x-y|^\mu}dydx\leq C(t,\mu,r)\|g\|_t\|h\|_r.
\end{equation*}
\end{Lem}
Using Hardy-Littlewood-Sobolev inequality, it is easy to check that
\begin{equation}\label{eq2.3}
\|u\|_D\leq C\|u\|^2~~\text{and}~~\int_{\mathbb{R}^3}\phi_u u^2dx\leq C\|u\|_{\frac{12}{5}}^4\leq C\|u\|^4.
\end{equation}
Therefore, the functional $I_\varepsilon$ is well-defined for every $u\in H^1(\mathbb{R}^3)$ and belongs to $C^1(H^1(\mathbb{R}^3),\mathbb{R})$. Moreover, for any $u, v\in H^1(\mathbb{R}^3)$, we have
\begin{equation*}
\langle I_\varepsilon^\prime(u),v\rangle=\int_{\mathbb{R}^3}\nabla u\nabla vdx+\int_{\mathbb{R}^3}V(\varepsilon x)uvdx+\int_{\mathbb{R}^3}\phi_u uvdx-\int_{\mathbb{R}^3}K(\varepsilon x)f(u)vdx,
\end{equation*}
where $\langle \cdot ,\cdot \rangle$ denotes the usual duality.

\subsection{Penalization argument}
In what follows, we will not work directly with the functional $I_\varepsilon$, because we have some difficulties to verify the $(PS)$ condition. We will adapt for our case an argument explored by the penalization method introduction by del Pino and Felmer \cite{DelPino-Felmer1996CVPDE}, and build a suitable modification of the energy functional $I_\varepsilon$ such that it satisfies the $(PS)$ condition.
\par
For any $\varepsilon>0$ and any $B\subset \mathbb{R}^3$, we define
\begin{equation*}
B_\epsilon=\{x\in \mathbb{R}^3|\varepsilon x\in B\}.
\end{equation*}
Let $\varsigma \in C^\infty(\mathbb{R})$ be a cut-off function such that $0\leq \varsigma(t) \leq1$ and $\varsigma^\prime(t)\geq0$ for every $t\in \mathbb{R},
\varsigma(t) = 0$ if $t\leq0, \varsigma(t) > 0$ if $t > 0$ and $\varsigma(t) = 1$ if $t\geq 1$. Let
\begin{equation*}
\chi_\varepsilon(x)=
\begin{cases}
0,&\text{if}~x\in \Lambda_\varepsilon,\\
\varepsilon^{-6}\varsigma(\text{dist}(x,\Lambda_\varepsilon)),&\text{if}~x\notin \Lambda_\varepsilon.
\end{cases}
\end{equation*}
It is easy to see that $\chi_\varepsilon$ is a $C^1$ function for $\varepsilon$ small and
\begin{equation*}
\chi_\varepsilon(x)=0~~\text{if}~~x\in \Lambda_\varepsilon,~~~\chi_\varepsilon(x)=\varepsilon^{-6}~~\text{if}~~x\notin (\Lambda_\varepsilon)^1.
\end{equation*}
For $u\in H^1(\mathbb{R}^3)$, we introduce the penalization term
\begin{equation*}
Q_\varepsilon(u)=\bigg(\int_{\mathbb{R}^3}\chi_\varepsilon(x)u^2dx-1\bigg)_+^\beta
\end{equation*}
where $2<\beta<\frac{\mu}{2}$ and $(t)_+=\max\{t,0\}$.
\par
Then we are ready to define the modified functional with penalization term by:
\begin{equation}\label{eq2.4}
\Phi_\varepsilon(u)=\frac{1}{2}\int_{\mathbb{R}^3}|\nabla u|^2dx+\frac{1}{2}\int_{\mathbb{R}^3}V(\varepsilon x)u^2dx+\frac{1}{4}\int_{\mathbb{R}^3}\phi_u u^2dx+Q_\varepsilon(u)-\int_{\mathbb{R}^3}K(\varepsilon x)F(u)dx.
\end{equation}
It is easy to see that the critical point of $\Phi_\varepsilon$ is a solution of
\begin{equation}\label{eq2.5}
\begin{cases}
-\Delta u+V(\varepsilon x)u+\phi u
+2\beta\big(\int_{\mathbb{R}^3}\chi_\varepsilon(x)u^2dx-1\big)_+^{\beta-1}\chi_\varepsilon(x)u=K(\varepsilon x)f(u),\\
-\Delta \phi=u^2.
\end{cases}
\end{equation}
And if $u$ is a critical point of $\Phi_\epsilon$ with $Q_\epsilon(u) = 0$, then $u$ is a solution of system \eqref{eq2.1}.

\begin{Lem}\label{Lem2.2}
$\Phi_\varepsilon$ satisfies the $(PS)_c$ condition.
\end{Lem}
\begin{proof}
Let $\{u_n\}$ be a $(PS)_c$ sequence of the functional $\Phi_\varepsilon$, that is
\begin{equation*}
\Phi_\varepsilon(u_n)\rightarrow c~~\text{and}~~\Phi_\varepsilon^\prime(u_n)\rightarrow 0~~\text{as}~n\rightarrow\infty.
\end{equation*}
Therefore, it follows from $(f_3)$ and the fact $\phi_{u_n}(x)\geq 0$ for all $x\in \mathbb{R}^3$ that
\begin{equation*}
\begin{split}
c+1+\|u_n\|&\geq\Phi_\varepsilon(u_n)-\frac{1}{\mu}\langle \Phi_\varepsilon^\prime(u_n),u_n\rangle\\
&=\frac{\mu-2}{2\mu}\bigg(\int_{\mathbb{R}^3}|\nabla u_n|^2dx+\int_{\mathbb{R}^3}V(\varepsilon x)u_n^2dx\bigg)+\frac{\mu-4}{4\mu}\int_{\mathbb{R}^3}\phi_{u_n}u_n^2dx+\big(\int_{\mathbb{R}^3}\chi_\varepsilon(x)u_n^2dx-1\big)_+^\beta\\
&\quad-\frac{2\beta}{\mu}\bigg(\int_{\mathbb{R}^3}\chi_\varepsilon(x)u_n^2dx-1\bigg)_+^{\beta-1}
\int_{\mathbb{R}^3}\chi_\varepsilon(x)u_n^2dx+\int_{\mathbb{R}^3}K(\varepsilon x)[\frac{1}{\mu}f(u_n)u_n-F(u_n)]dx\\
&\geq \frac{\min\{1,a_1\}(\mu-2)}{2\mu}\|u_n\|^2+\big(\int_{\mathbb{R}^3}\chi_\varepsilon(x)u_n^2dx-1\big)_+^\beta\\
&\quad-\frac{2\beta}{\mu}\bigg(\int_{\mathbb{R}^3}\chi_\varepsilon(x)u_n^2dx-1\bigg)_+^{\beta-1}
\int_{\mathbb{R}^3}\chi_\varepsilon(x)u_n^2dx\\
&\geq \frac{\min\{1,a_1\}(\mu-2)}{2\mu}\|u_n\|^2+\frac{\mu-2\beta}{\mu}\big(\int_{\mathbb{R}^3}\chi_\varepsilon(x)u_n^2dx-1\big)_+^\beta
-\frac{2\beta}{\mu}\bigg(\int_{\mathbb{R}^3}\chi_\varepsilon(x)u_n^2dx-1\bigg)_+^{\beta-1}
\end{split}
\end{equation*}
here we have used the fact that
\begin{equation*}
\int_{\mathbb{R}^3}\chi_\varepsilon(x)u_n^2dx\leq\big(\int_{\mathbb{R}^3}\chi_\varepsilon(x)u_n^2dx-1\big)_++1.
\end{equation*}
Since $2<\beta<\frac{\mu}{2}$, we get $\{u_n\}$ is bounded in $H^1(\mathbb{R}^3)$ and $\big(\int_{\mathbb{R}^3}\chi_\varepsilon(x)u_n^2dx-1\big)_+^\beta$ is bounded.
We assume that, up to a subsequence, $u_n\rightharpoonup u$ in $H^1(\mathbb{R}^3)$ as $n\to \infty$ and
\begin{equation}\label{eq2.6}
\lambda_n:=2 \beta\big(\int_{\mathbb{R}^3}\chi_\varepsilon(x)u_n^2dx-1\big)_+^{\beta-1}\rightarrow \lambda,~\text{as}~n\rightarrow \infty.
\end{equation}
It is easy to verify that $u$ solves
\begin{equation*}
-\Delta u+V(\varepsilon x)u+\phi_u u+\lambda \chi_\varepsilon(x)u=K(\varepsilon x)f(u).
\end{equation*}
Therefore, for any $v\in H^1(\mathbb{R}^3)$
\begin{equation}\label{eq2.7}
\begin{split}
o_n(\|v\|)=\langle \Phi_\varepsilon^\prime(u_n),v\rangle&=\int_{\mathbb{R}^3}\nabla (u_n-u)\nabla v dx+\int_{\mathbb{R}^3}V(\varepsilon x)(u_n-u)vdx+
\int_{\mathbb{R}^3}(\phi_{u_n}u_n-\phi_uu)vdx\\
&\quad+\lambda\int_{\mathbb{R}^3}\chi_\varepsilon(x)(u_n-u)vdx+(\lambda_n-\lambda)\int_{\mathbb{R}^3}\chi_\varepsilon(x)u_nvdx\\
&\quad-\int_{\mathbb{R}^3}K(\varepsilon x)(f(u_n)-f(u))vdx.
\end{split}
\end{equation}
Then we have as $n, m\rightarrow\infty$,
\begin{equation}\label{eq2.8}
\begin{split}
o_{m,n}(\|u_n-u_m\|)&=\langle \Phi_\varepsilon^\prime(u_n)-\Phi_\varepsilon^\prime(u_m),u_n-u_m\rangle\\
&=\int_{\mathbb{R}^3}|\nabla(u_n-u_m)|^2dx+\int_{\mathbb{R}^3}V(\varepsilon x)|u_n-u_m|^2dx+\int_{\mathbb{R}^3}(\phi_{u_n}u_n-\phi_{u_m}u_m)(u_n-u_m)dx
\\
&\quad+\int_{\mathbb{R}^3}(\lambda_n u_n-\lambda_mu_m)(u_n-u_m)\chi_\varepsilon(x)dx+\int_{\mathbb{R}^3}K(\varepsilon x)(f(u_m)-f(u_n))(u_n-u_m)dx\\
&\geq \min\{1,a_1\}\|u_n-u_m\|^2+\int_{\mathbb{R}^3}(\lambda_n u_n-\lambda_mu_m)(u_n-u_m)\chi_\varepsilon(x)dx\\
&\quad+\int_{\mathbb{R}^3}(\phi_{u_n}-\phi_{u_m})u_m(u_n-u_m)dx+\int_{\mathbb{R}^3}K(\varepsilon x)(f(u_m)-f(u_n))(u_n-u_m)dx
\end{split}
\end{equation}
here we have used the fact that
\begin{equation*}
\begin{split}
\int_{\mathbb{R}^3}(\phi_{u_n}u_n-\phi_{u_m}u_m)(u_n-u_m)dx&
=\int_{\mathbb{R}^3}\phi_{u_n}(u_n-u_m)^2dx+\int_{\mathbb{R}^3}(\phi_{u_n}-\phi_{u_m})u_m(u_n-u_m)dx\\
&\geq \int_{\mathbb{R}^3}(\phi_{u_n}-\phi_{u_m})u_m(u_n-u_m)dx.
\end{split}
\end{equation*}
From \eqref{eq2.6}, we get that as $n,m\rightarrow\infty$
\begin{equation}\label{eq2.9}
\int_{\mathbb{R}^3}(\lambda_n u_n-\lambda_mu_m)(u_n-u_m)\chi_\varepsilon(x)dx=\lambda\int_{\mathbb{R}^3}(u_n-u_m)^2\chi_\varepsilon(x)dx+o(1).
\end{equation}
\par
Let $r_0 > 0$ be such that $\Lambda\subset B_{r_0}(0)$. Then $\Lambda_\varepsilon\subset B_{\varepsilon^{-1}r_0+1}(0)$ and by the boundedness of $\big(\int_{\mathbb{R}^3}\chi_\varepsilon(x)u_n^2dx-1\big)_+^\beta$, one has
\begin{equation*}
\begin{split}
\int_{\{|x|\geq \varepsilon^{-1}r_0+1\}}u_n^2dx&\leq\int_{\mathbb{R}^3\setminus (\Lambda_\varepsilon)^1}u_n^2dx=\varepsilon^6\int_{\mathbb{R}^3\setminus (\Lambda_\varepsilon)^1}\chi_\varepsilon(x)u_n^2dx\leq \varepsilon^6\int_{\mathbb{R}^3}\chi_\varepsilon(x)u_n^2dx\leq C\varepsilon^6.
\end{split}
\end{equation*}
Using the interpolation inequality $\|u\|_p\leq \|u\|_2^t\|u\|_6^{1-t}\leq C\|u\|_2^t\|u\|^{1-t}$, where the positive constant $C$ is independent of $n ,\varepsilon$ and $t=\frac{6-3p}{2p}$. Thus,
\begin{equation}\label{eq2.10}
\int_{\{|x|\geq \varepsilon^{-1}r_0+1\}}|u_n|^pdx\leq C\varepsilon^{\frac{3(6-p)}{2}}.
\end{equation}
Using the mean value theorem, we get that there exists $0 <\theta(x) <1$ such that
\begin{equation}\label{eq2.11}
\begin{split}
&\big|\int_{\mathbb{R}^3}K(\varepsilon x)(f(u_m)-f(u_n))(u_n-u_m)dx\big|\\
&=\big|\int_{\mathbb{R}^3}K(\varepsilon x)f^\prime(\theta u_n+(1-\theta)u_m)(u_n-u_m)^2dx\big|\\
&\leq \delta\int_{\mathbb{R}^3}|u_n-u_m|^2dx+C(\delta)\int_{\mathbb{R}^3}(|u_n|^{p-2}+|u_m|^{p-2})|u_n-u_m|^2dx\\
&\leq \frac{1}{2}\|u_n-u_m\|^2+C\int_{\{|x|\geq \varepsilon^{-1}r_0+1\}}(|u_n|^{p-2}+|u_m|^{p-2})|u_n-u_m|^2dx\\
&\quad+C\int_{\{|x|\leq \varepsilon^{-1}r_0+1\}}(|u_n|^{p-2}+|u_m|^{p-2})|u_n-u_m|^2dx.
\end{split}
\end{equation}
Since $u_n\rightharpoonup u$ in $H^1(\mathbb{R}^3)$, we get $u_n\rightarrow u$ in $L^p(\{x\in \mathbb{R}^3||x|\leq \epsilon^{-1}r_0+1\})$. It follows that as $n,m\rightarrow\infty$
\begin{equation}\label{eq2.12}
\begin{split}
&\int_{\{|x|\leq \varepsilon^{-1}r_0+1\}}(|u_n|^{p-2}+|u_m|^{p-2})|u_n-u_m|^2dx\\
&\leq\bigg(\big(\int_{\{|x|\leq\varepsilon^{-1}r_0+1\}}|u_n|^pdx\big)^{\frac{p-2}{p}}+\big(\int_{\{|x|\leq\varepsilon^{-1}r_0+1\}}|u_m|^pdx\big)^{\frac{p-2}{p}}\bigg)\bigg(\int_{\{|x|\leq \varepsilon^{-1}r_0+1\}}|u_n-u_m|^pdx\bigg)^{\frac{2}{p}}\\
&=o(1).
\end{split}
\end{equation}
And from \eqref{eq2.10}, we get that
\begin{equation}\label{eq2.13}
\begin{split}
&\int_{\{|x|\geq\varepsilon^{-1}r_0+1\}}(|u_n|^{p-2}+|u_m|^{p-2})|u_n-u_m|^2dx\\
&\leq\bigg(\big(\int_{\{|x|\geq\varepsilon^{-1}r_0+1\}}|u_n|^pdx\big)^{\frac{p-2}{p}}+\big(\int_{\{|x|\geq\varepsilon^{-1}r_0+1\}}|u_m|^pdx\big)^{\frac{p-2}{p}}\bigg)\bigg(\int_{\{|x|\geq \varepsilon^{-1}r_0+1\}}|u_n-u_m|^pdx\bigg)^{\frac{2}{p}}\\
&\leq C\varepsilon^{\frac{3(6-p)(p-2)}{2p}} \|u_n-u_m\|^2.
\end{split}
\end{equation}
Combining with \eqref{eq2.11}-\eqref{eq2.13}, we get that, as $n, m\rightarrow \infty$,
\begin{equation}\label{eq2.14}
|\int_{\mathbb{R}^3}K(\varepsilon x)(f(u_m)-f(u_n))(u_n-u_m)dx|\leq (\frac{1}{2}+C\varepsilon^{\frac{3(6-p)(p-2)}{2p}})\|u_n-u_m\|^2+o(1).
\end{equation}
By Hardy-Littlewood-Sobolev inequality
\begin{equation}\label{eq2.15}
\begin{split}
&\quad\int_{\mathbb{R}^3}(\phi_{u_n}-\phi_{u_m})u_m(u_n-u_m)dx\\
&=\int_{\mathbb{R}^3}\int_{\mathbb{R}^3}\frac{(u_n^2(y)-u_m^2(y))u_m(x)(u_n(x)-u_m(x))}{|x-y|}dxdy\\
&\leq \big(\int_{\mathbb{R}^3}|u_n^2-u_m^2|^{\frac{6}{5}}dx\big)^{\frac{5}{6}}\big(\int_{\mathbb{R}^3}|u_m|^{\frac{6}{5}}|u_n-u_m|^{\frac{6}{5}}dx\big)^{\frac{5}{6}}\\
&\leq \big(\int_{\mathbb{R}^3}|u_n^2-u_m^2|^{\frac{6}{5}}dx\big)^{\frac{5}{6}}\bigg[
\big(\int_{\{|x|\geq \varepsilon^{-1}r_0+1\}}|u_m|^{\frac{6}{5}}|u_n-u_m|^{\frac{6}{5}}dx\big)^{\frac{5}{6}}
+\big(\int_{\{|x|\leq \varepsilon^{-1}r_0+1\}}|u_m|^{\frac{6}{5}}|u_n-u_m|^{\frac{6}{5}}dx\big)^{\frac{5}{6}}\bigg]\\
&\leq \big(\int_{\mathbb{R}^3}|u_n^2-u_m^2|^{\frac{6}{5}}dx\big)^{\frac{5}{6}}\big(\int_{\{|x|\geq \varepsilon^{-1}r_0+1\}}|u_m|^{\frac{12}{5}}dx\big)^{\frac{5}{12}}\big(\int_{\mathbb{R}^3}|u_n-u_m|^{\frac{12}{5}}dx\big)^{\frac{5}{12}}\\
&\quad+\big(\int_{\mathbb{R}^3}|u_n^2-u_m^2|^{\frac{6}{5}}dx\big)^{\frac{5}{6}}\big(\int_{\mathbb{R}^3}
|u_m|^{\frac{12}{5}}dx\big)^{\frac{5}{12}}\big(\int_{\{|x|\leq \varepsilon^{-1}r_0+1\}}|u_n-u_m|^{\frac{12}{5}}dx\big)^{\frac{5}{12}}\\
&\leq C\varepsilon^{\frac{9}{4}} \big(\int_{\mathbb{R}^3}|u_n-u_m|^{\frac{12}{5}}dx\big)^{\frac{5}{6}}+o(1)\\
&\leq C\varepsilon^{\frac{9}{4}} \|u_n-u_m\|^2+o(1).
\end{split}
\end{equation}
So, \eqref{eq2.8}-\eqref{eq2.9} and \eqref{eq2.14}-\eqref{eq2.15} imply that $\{u_n\}$ is a Cauchy sequence in $H^1(\mathbb{R}^3)$, hence a convergent sequence.
\end{proof}

\section{Existence of multiple sign-changing solutions for modified problem}
In this section, we construct multiple sign-changing critical points of the modified functionals $\Phi_\varepsilon$ . For this purpose, we adapt an abstract critical point theorem in \cite{Liu-Liu-Wang2016JDE}. For reader's convenience, we state it here.
\par
Let $X$ be a Banach space, $J$ be an even $C^1$ functional on $X$. Let $P,Q$ be open convex sets of $X,Q=-P$. Set
\begin{equation*}
O=P\cup Q,~~\Sigma =\partial P\cap \partial Q.
\end{equation*}
Assume
\begin{itemize}
\item[$(I_1)$] $J$ satisfies the $(PS)$ condition.
\item[$(I_2)$] $c_*= \inf\limits_{u\in \Sigma} J (u)> 0$.
\end{itemize}
Assume there exists an odd locally Lipschitz continuous map $A : X\rightarrow X$ satisfying:
\begin{itemize}
\item[$(A_1)$] Given $c_0,b_0 >0$, there exists $b=b(c_0,b_0)>0$ such that if $\|J^\prime(u)\|\geq b_0,|J(u)|\leq c_0$, then
\begin{equation*}
\langle J^\prime (u),u-Au\rangle\geq b\|u-Au\|>0.
\end{equation*}
\item[$(A_2)$] $A(\partial P)\subset P,~A(\partial Q)\subset Q$.
\end{itemize}
Define
\begin{equation*}
\Theta=\{\eta|\eta \in C(X,X),\eta~\text{odd},~\eta(P)\subset P,\eta(Q)\subset Q,\eta(u)=u~\text{if}~J(u)<0\},
\end{equation*}
\begin{equation*}
\Gamma_j=\{E|E\subset X,E~\text{compact},-E=E,\gamma(E\cap \eta^{-1}(\Sigma))\geq j~\text{for}~\eta\in \Theta\},
\end{equation*}
where $\gamma$ be the genus of symmetric sets
\begin{equation*}
\gamma(E)=\inf\{n|\text{there~exists~an~odd~map}~\varphi:E\rightarrow\mathbb{R}^n\setminus\{0\}\}.
\end{equation*}
Assume that
\begin{itemize}
\item[$(\Gamma)$] $\Gamma_j$ is nonempty, $j=1,2,\cdots$.
\end{itemize}
Define
\begin{equation*}
c_j=\inf_{A\in \Gamma_j}\sup_{u\in A\setminus O}J(u),~j=1,2,\cdots,
\end{equation*}
\begin{equation*}
K_c=\{u|J^\prime(u)=0,J(u)=c\},~K_c^*=K_c\setminus O.
\end{equation*}
The following abstract critical point theorem is from \cite{Liu-Liu-Wang2016JDE}.
\begin{Thm}\label{Thm3.1}
Assume $(I_1), (I_2), (A_1), (A_2)$ and $(\Gamma)$ hold. Then
\begin{itemize}
\item[$(1.)$] $c_j\geq c_*, K_{c_j}^*\neq \emptyset$.
\item[$(2.)$] $c_j\rightarrow\infty$ as $j\rightarrow\infty$.
\item[$(3.)$] if $c_j=c_{j+1}=\cdots=c_{j+k-1}=c$, then $\gamma(K_c^*)\geq k$.
\end{itemize}
\end{Thm}
\par
In the following, we verify that the functional $\Phi_\varepsilon$ satisfies all the assumptions of Theorem \ref{Thm3.1}. In Lemma \ref{Lem2.2} we have proved that $\Phi_\varepsilon$ satisfies the assumption $(I_1)$, i.e. the $(PS)$ condition.
\par
Now we introduce an auxiliary operator $A_\varepsilon$. Precisely, the operator $A_\varepsilon$ is defined as follows: for any $u\in H^1(\mathbb{R}^3), v = A_\varepsilon(u)\in H^1(\mathbb{R}^3)$ is the unique solution to the equation
\begin{equation}\label{eq3.1}
-\Delta v+V(\varepsilon x)v+\phi_uv+2\beta \kappa(u)\chi_\varepsilon(x) v=K(\varepsilon x)f(u),~v\in H^1(\mathbb{R}^3)£¬
\end{equation}
where $\kappa(u)=\big(\int_{\mathbb{R}^3}\chi_\varepsilon(x)u^2dx-1\big)_+^{\beta-1}$. Clearly, the three statements are equivalent: $u$ is a solution of equation \eqref{eq3.1}, $u$ is a critical point of $\Phi_\varepsilon$, and $u$ is a fixed point of $A_\varepsilon$.
\par
\begin{Lem}\label{Lem3.1}
The operator $A_\varepsilon$ is well defined and is locally Lipschitz continuous on $H^1(\mathbb{R}^3)$.
\end{Lem}
\begin{proof}
We first prove $v$ can be obtained by solving the following minimization problem:
\begin{equation*}
\inf\{J_\varepsilon(v):v\in H^1(\mathbb{R}^3)\},
\end{equation*}
where
\begin{equation*}
\begin{split}
J_\varepsilon(v)=&\frac{1}{2}\int_{\mathbb{R}^3}|\nabla v|^2dx+\frac{1}{2}\int_{\mathbb{R}^3}V(\varepsilon x)v^2dx+\frac{1}{2}\int_{\mathbb{R}^3}\phi_uv^2dx\\
&+2\beta\kappa(u)\int_{\mathbb{R}^3}\chi_\varepsilon(x)v^2dx-\int_{\mathbb{R}^3}K(\varepsilon x)f(u)vdx.
\end{split}
\end{equation*}
In fact, by \eqref{eq1.9}, we have
\begin{equation*}
\begin{split}
J_\varepsilon(v)&\geq\frac{1}{2}\int_{\mathbb{R}^3}|\nabla v|^2dx+\frac{1}{2}\int_{\mathbb{R}^3}V(\varepsilon x)v^2dx-C\int_{\mathbb{R}^3}(|u|+|u|^{p-1})|v|dx\\
&\geq \min\{\frac{1}{2},\frac{a_1}{2}\}\|v\|^2-C\|u\|_2\|v\|_2-C\|u\|_p^{p-1}\|v\|_p\\
&\geq \min\{\frac{1}{2},\frac{a_1}{2}\}\|v\|^2-C\|v\|,
\end{split}
\end{equation*}
which deduces that $J_\varepsilon$ is coercive and weakly lower semicontinuous. Next we prove that $v$ is unique. Assume $v_1, v_2$ are two solutions corresponding to $u$, then we have
\begin{equation*}
\begin{split}
\langle J_\varepsilon^\prime(v_1)-J_\varepsilon^\prime(v_2),v_1-v_2\rangle&=
\int_{\mathbb{R}^3}|\nabla (v_1-v_2)|^2dx+\int_{\mathbb{R}^3}V(\varepsilon x)(v_1-v_2)^2dx\\
&\quad+\int_{\mathbb{R}^3}\phi_u(v_1-v_2)^2dx+2\beta\kappa(u)\int_{\mathbb{R}^3}\chi_\varepsilon(x)(v_1-v_2)^2dx\\
&\geq \min\{\frac{1}{2},\frac{a_1}{2}\}\|v_1-v_2\|^2.
\end{split}
\end{equation*}
Thus we have $v_1 = v_2$.
\par
Clearly, $A_\varepsilon$ maps bounded sets into bounded sets.  Now we will show that the map $A_\varepsilon$ is locally Lipschitz continuous. Let $\bar{u}=A_\varepsilon(u)$ and $\bar{v}=A_\varepsilon(v)$, we have
\begin{equation*}
\begin{split}
\min\{1,a_1\}\|\bar{u}-\bar{v}\|^2&\leq \int_{\mathbb{R}^3}|\nabla(\bar{u}-\bar{v})|^2dx+\int_{\mathbb{R}^3}V(\varepsilon x)|\bar{u}-\bar{v}|^2dx\\
&=\int_{\mathbb{R}^3}K(\varepsilon x)(f(u)-f(v))(\bar{u}-\bar{v})dx-\int_{\mathbb{R}^3}(\phi_u \bar{u}-\phi_v \bar{v})(\bar{u}-\bar{v})dx\\
&\quad-2\beta \kappa(u)\int_{\mathbb{R}^3}\chi_\varepsilon(x)\bar{u}(\bar{u}-\bar{v})dx+2\beta \kappa(v)\int_{\mathbb{R}^3}\chi_\varepsilon(x)\bar{v}(\bar{u}-\bar{v})dx\\
&\leq
\int_{\mathbb{R}^3}K(\varepsilon x)(f(u)-f(v))(\bar{u}-\bar{v})dx+\int_{\mathbb{R}^3}(\phi_v-\phi_u)\bar{v}(\bar{u}-\bar{v})dx\\
&\quad+2\beta\big(\kappa(v)-\kappa(u)\big)\int_{\mathbb{R}^3}\chi_\varepsilon(x)\bar{v}(\bar{u}-\bar{v})dx\\
&=\Pi_1+\Pi_2+\Pi_3.
\end{split}
\end{equation*}
Using the mean value theorem, we get that there exists $0 <\theta(x) <1$ such that
\begin{equation}\label{eq3.2}
\begin{split}
\Pi_1&\leq C\big|\int_{\mathbb{R}^3}(f(u)-f(v))(\bar{u}-\bar{v})dx\big|=C\big|\int_{\mathbb{R}^3}f^\prime(\theta u+(1-\theta)v)(u-v)(\bar{u}-\bar{v})dx\big|\\
&\leq C\int_{\mathbb{R}^3}|u-v||\bar{u}-\bar{v}|dx+C\int_{\mathbb{R}^3}(|u|^{p-2}+|v|^{p-2})|u-v||\bar{u}-\bar{v}|dx\\
&\leq C\|u-v\|_2\|\bar{u}-\bar{v}\|_2+C(\|u\|_p^{p-2}+\|v\|_p^{p-2})\|u-v\|_p\|\bar{u}-\bar{v}\|_p\\
&\leq C\|u-v\|\|\bar{u}-\bar{v}\|.
\end{split}
\end{equation}
Next, we estimate the second term $\Pi_2$. By Hardy-Littlewood-Sobolev inequality and H\"older inequality, one has
\begin{equation}\label{eq3.3}
\begin{split}
\Pi_2&\leq \big|\int_{\mathbb{R}^3}(\phi_v-\phi_u)\bar{v}(\bar{u}-\bar{v})dx\big|\\
&\leq\|u-v\|_{\frac{12}{5}}\|u+v\|_{\frac{12}{5}}\|\bar{v}\|_{\frac{12}{5}}\|\bar{u}-\bar{v}\|_{\frac{12}{5}}\\
&\leq C\|u-v\|\|u+v\|\|\bar{v}\|\|\bar{u}-\bar{v}\|\\
&\leq C\|u-v\|\|\bar{u}-\bar{v}\|.
\end{split}
\end{equation}
Now, we estimate the third term $\Pi_3$. Using the elementary inequality $|a^\alpha-b^\alpha|\leq \alpha \max\{a^{\alpha-1},b^{\alpha-1}\}|a-b|$, which holds for $\alpha \geq 1$ and $a, b\geq 0$, we have
\begin{equation*}
\begin{split}
|\kappa(v)-\kappa(u)|&\leq (\beta -1)\max\bigg\{\big(\int_{\mathbb{R}^3}\chi_\varepsilon(x)u^2dx-1\big)_+^{\beta-2},\big(\int_{\mathbb{R}^3}\chi_\varepsilon(x)v^2dx-1\big)_+^{\beta-2}\bigg\}\cdot\\
&\quad\bigg|\big(\int_{\mathbb{R}^3}\chi_\varepsilon(x)u^2dx-1\big)_+-\big(\int_{\mathbb{R}^3}\chi_\varepsilon(x)v^2dx-1\big)_+\bigg|\\
&\leq C\int_{\mathbb{R}^3}\chi_\varepsilon(x)|u^2-v^2|dx\\
&\leq C\|u-v\|.
\end{split}
\end{equation*}
Thus,
\begin{equation}\label{eq3.4}
\Pi_3\leq C\|u-v\|\|\bar{v}\|\|\bar{u}-\bar{v}\|\leq C\|u-v\|\|\bar{u}-\bar{v}\|.
\end{equation}
From \eqref{eq3.2}-\eqref{eq3.4}, we deduce the desired result.
\end{proof}

Let
\begin{equation*}
P^+:=\{u\in H^1(\mathbb{R}^3)|u\geq 0\}~~\text{and}~~P^-:=\{u\in H^1(\mathbb{R}^3)|u\leq 0\}.
\end{equation*}
For an arbitrary $\sigma>0$, we define
\begin{equation*}
P_\sigma^+:=\{u\in H^1(\mathbb{R}^3)|\text{dist}(u,P^+)<\sigma\}~~\text{and}~~P_\sigma^-:=\{u\in H^1(\mathbb{R}^3)|\text{dist}(u,P^-)<\sigma\},
\end{equation*}
where $\text{dist}(u,P^{\pm})=\inf\limits_{w\in P^\pm}\|u-w\|$. Obviously, $P_\sigma^-=-P_\sigma^+$, and $O=P_\sigma^+\cup P_\sigma^-$ is a open and symmetric subset of $H^1(\mathbb{R}^3)$ and $H^1(\mathbb{R}^3)\setminus O$ contains only sign-changing function.
\par
We verify the assumption $(I_2)$ of Theorem \ref{Thm3.1}.
\begin{Lem}\label{Lem3.2}
There exists $\sigma^\prime>0$ such that for $\sigma\in(0,\sigma^\prime)$, there holds
\begin{equation*}
\Phi_\varepsilon(u)\geq \frac{\min\{1,a_1\}}{4}\sigma^2~\text{for}~u\in \Sigma=\partial P_\sigma^+\cap \partial P_\sigma^-,
\end{equation*}
and then $c_\varepsilon^*:=\inf\limits_{u\in\Sigma}\Phi_\varepsilon(u)\geq \frac{\min\{1,a_1\}}{4}\sigma^2$.
\end{Lem}
\begin{proof}
For any $u\in \partial P_\sigma^+\cap \partial P_\sigma^-$, there exist $c_p>0$ such that
\begin{equation*}
\|u^\pm\|_p=\inf_{w\in P^\mp}\|u-w\|_p\leq c_p\inf_{w\in P^\mp}\|u-w\|=c_p\text{dist}(u,P^\mp)=c_p\sigma,
\end{equation*}
which implies that $\|u\|_p\leq 2c_p\sigma$. Clearly, $\|u^\pm\|\geq \text{dist}(u,P^\mp)=\sigma$. Thus, by \eqref{eq1.9}, one has
\begin{equation*}
\begin{split}
\Phi_\varepsilon(u)&\geq \frac{1}{2}\int_{\mathbb{R}^3}|\nabla u|^2dx+\frac{1}{2}\int_{\mathbb{R}^3}V(\varepsilon x)u^2dx-\int_{\mathbb{R}^3}K(\varepsilon x)F(u)dx\\
&\geq \frac{\min\{1,a_1\}}{2}\|u\|^2-C\|u\|_p^p\\
& \geq \frac{\min\{1,a_1\}}{2}\sigma^2-C\sigma^p\\
&\geq \frac{\min\{1,a_1\}}{4}\sigma^2,
\end{split}
\end{equation*}
for $\sigma$ small enough, and the proof is completed.
\end{proof}
\par
Now, we verify the assumption $(A_2)$ of Theorem \ref{Thm3.1}.
\begin{Lem}\label{Lem3.3}
There exists $0<\sigma_0<\sigma^\prime$ such that for $\sigma\in(0,\sigma_0)$,
\begin{equation*}
A(\partial P_\sigma^-)\subset P_\sigma^-,~~A(\partial P_\sigma^+)\subset P_\sigma^+.
\end{equation*}
\end{Lem}
\begin{proof}
Since the two conclusions are similar, we only prove the first one. Let $u\in H^1(\mathbb{R}^3)$ and $v=A_\varepsilon(u)$. It follows from \eqref{eq1.9} and the fact $\text{dist}(v,P^-)\leq\|v^+\|$ that
 \begin{equation*}
 \begin{split}
\text{dist}(v,P^-)\|v^+\|&\leq\|v^+\|^2\leq \min\{1,a_1\}^{-1}\bigg(\int_{\mathbb{R}^3}\nabla v \nabla v^+dx+\int_{\mathbb{R}^3}V(\varepsilon x)vv^+dx\bigg)\\
&=\min\{1,a_1\}^{-1}\bigg(\int_{\mathbb{R}^3}K(\varepsilon x)f(u)v^+dx-2\beta\kappa(u)
\int_{\mathbb{R}^3}\chi_\varepsilon(x)vv^+dx-\int_{\mathbb{R}^3}\phi_u vv^+dx\bigg)\\
&\leq C\int_{\mathbb{R}^3}f(u)v^+dx\\\
&= C\int_{\mathbb{R}^3}f(u^+)v^+dx\\
&\leq \int_{\mathbb{R}^3}\big(\delta|u^+|+C(\delta)|u^+|^{p-1}\big)|v^+|dx\\
&\leq \delta \|u^+\|_2\|v^+\|_2+C(\delta)\|u^+\|_p^{p-1}\|v^+\|_p\\
&\leq C\big(\delta \text{dist}(u,P^-)+C(\delta) \text{dist}(u,P^-)^{p-1}\big)\|v^+\|.
\end{split}
\end{equation*}
It follows that
\begin{equation*}
\text{dist}(A(u),P^-)\leq C\big(\delta \text{dist}(u,P^-)+C(\delta)\text{dist}(u,P^-)^{p-1}\big).
\end{equation*}
Thus, choosing $\delta$ small enough, there exists $\sigma_0>0$ such that for $\sigma\in(0,\sigma_0)$,
\begin{equation*}
\text{dist}(A(u),P^-)\leq\frac{1}{2}\text{dist}(u,P^-)~~\text{for any}~~u\in P_\sigma^-.
\end{equation*}
This implies that $A(\partial P_\sigma^-)\subset P_\sigma^-$.
\end{proof}
\par
To verify the assumption $(A_1)$, we need the following Lemma.
\begin{Lem}\label{Lem3.4}
For any $u\in H^1(\mathbb{R}^3)$, one has
\begin{equation}\label{eq3.5}
\langle\Phi_\varepsilon^\prime(u),u-A_\varepsilon(u)\rangle\geq\min\{1,a_1\}\|u-A_\varepsilon(u)\|^2.
\end{equation}
Moreover, there exists $C>0$ such that
\begin{equation}\label{eq3.6}
 \|\Phi_\varepsilon^\prime(u)\|\leq \|u-A_\varepsilon(u)\|(1+C\|u\|^{2\beta-2}).
 \end{equation}
\end{Lem}
\begin{proof}
Since $A_\varepsilon(u)$ is the solution of equation \eqref{eq3.1}, by a direct computation, we see that
 \begin{equation*}
 \begin{split}
\langle \Phi_\varepsilon^\prime(u),u-A_\varepsilon(u)\rangle&=\int_{\mathbb{R}^3}|\nabla (u-A_\varepsilon(u))|^2dx+\int_{\mathbb{R}^3}V(\varepsilon x)| u-A_\varepsilon(u)|^2dx+\int_{\mathbb{R}^3}\phi_u(u-A_\varepsilon(u))^2dx\\
&\quad+2\beta\kappa(u)\int_{\mathbb{R}^3}\chi_\varepsilon(x)(u-A_\varepsilon(u))^2dx\\
&\geq \min\{1,a_1\}\|u-A_\varepsilon(u)\|^2.
\end{split}
\end{equation*}
For all $\varphi\in H^1(\mathbb{R}^3)$, we have
\begin{equation*}
\begin{split}
\langle \Phi_\varepsilon^\prime(u),\varphi\rangle&=\int_{\mathbb{R}^3}\nabla (u-A_\varepsilon(u)) \nabla \varphi dx+\int_{\mathbb{R}^3}V(\varepsilon x)( u-A_\varepsilon(u))\varphi dx+\int_{\mathbb{R}^3}\phi_u(u-A_\varepsilon(u))\varphi dx\\
&\quad+2\beta\kappa(u)\int_{\mathbb{R}^3}\chi_\varepsilon(x)(u-A_\varepsilon(u))\varphi dx.
\end{split}
\end{equation*}
Note that
\begin{equation*}
\big|\int_{\mathbb{R}^3}\phi_u(u-A_\varepsilon(u))\varphi dx\big|\leq C\|u\|^2\|u-A_\varepsilon(u)\|\|\varphi\|,
\end{equation*}
and
\begin{equation*}
\big|\int_{\mathbb{R}^3}\chi_\varepsilon(x)(u-A_\varepsilon(u))\varphi dx\big|\leq C\|u-A_\varepsilon(u)\|\|\varphi\|.
\end{equation*}
Moreover,
\begin{equation*}
\kappa(u)=\big(\int_{\mathbb{R}^3}\chi_\varepsilon(x)u^2dx-1\big)_+^{\beta-1}\leq \big(2\int_{\mathbb{R}^3}\chi_\varepsilon(x)u^2dx\big)^{\beta-1}\leq C\|u\|^{2\beta-2}.
\end{equation*}
Thus, $\|\Phi_\varepsilon^\prime(u)\|\leq \|u-A_\varepsilon(u)\|(1+C\|u\|^{2\beta-2})$ for all $u\in H^1(\mathbb{R}^3)$.
\end{proof}
\begin{Lem}\label{Lem3.5}
Given $c_0,b_0 >0$, there exists $b=b(c_0,b_0)>0$ such that if $\|\Phi_\varepsilon^\prime(u)\|\geq b_0,|\Phi_\varepsilon(u)|\leq c_0$, then
\begin{equation*}
\langle \Phi_\varepsilon^\prime (u),u-A_\varepsilon (u)\rangle\geq b\|u-A_\varepsilon(u)\|>0.
\end{equation*}
\end{Lem}
\begin{proof}
By Lemma \ref{Lem3.4}, it suffices to prove that there exists $\beta_0>0$ such that
\begin{equation}\label{eq3.7}
\|u-A_\varepsilon (u)\|\geq \beta_0.
\end{equation}
For any $u\in H^1(\mathbb{R}^3)$, by $(f_3)$, we have
\begin{equation}\label{eq3.8}
\begin{split}
&\quad\Phi_\varepsilon(u)-\frac{1}{\mu}(u,u-A_\varepsilon(u))_\varepsilon\\
&=\frac{\mu -2}{2\mu}\big(\int_{\mathbb{R}^3}|\nabla u|^2 dx+\int_{\mathbb{R}^3}V(\varepsilon x)u^2dx\big)+
\frac{\mu-4}{4\mu}\int_{\mathbb{R}^3}\phi_uu^2dx\\
&\quad+\frac{1}{\mu}\int_{\mathbb{R}^3}\phi_uu(u-A_\varepsilon(u))dx+\int_{\mathbb{R}^3}K(\varepsilon x)(\frac{1}{\mu}f(u)u-F(u))dx\\
&\quad+\big(\int_{\mathbb{R}^3}\chi_\varepsilon(x)u^2dx-1\big)_+^\beta-\frac{2\beta}{\mu}\big(\int_{\mathbb{R}^3}\chi_\varepsilon(x)u^2dx-1\big)_+^{\beta-1}\int_{\mathbb{R}^3}\chi_\varepsilon(x)A_\varepsilon(u)udx\\
&\geq \frac{\min\{1,a_1\}(\mu -2)}{2\mu}\|u\|^2+
\frac{\mu-4}{4\mu}\int_{\mathbb{R}^3}\phi_uu^2dx+\frac{1}{\mu}\int_{\mathbb{R}^3}\phi_uu(u-A_\varepsilon(u))dx\\
&\quad+\big(\int_{\mathbb{R}^3}\chi_\varepsilon(x)u^2dx-1\big)_+^\beta-\frac{2\beta}{\mu}\big(\int_{\mathbb{R}^3}\chi_\varepsilon(x)u^2dx-1\big)_+^{\beta-1}\int_{\mathbb{R}^3}\chi_\varepsilon(x)u^2dx\\
&\quad+\frac{2\beta}{\mu}\big(\int_{\mathbb{R}^3}\chi_\varepsilon(x)u^2dx-1\big)_+^{\beta-1}\int_{\mathbb{R}^3}\chi_\varepsilon(x)(u-A_\varepsilon(u))udx.
\end{split}
\end{equation}
By Hardy-Littlewood-Sobolev inequality and H\"older inequality, one has
\begin{equation}\label{eq3.9}
\begin{split}
\bigg|\int_{\mathbb{R}^3}\phi_uu(u-A_\varepsilon(u))dx\bigg|&\leq \bigg(\int_{\mathbb{R}^3}\phi_u(u-A_\varepsilon(u))^2dx\bigg)^{\frac{1}{2}}\bigg(\int_{\mathbb{R}^3}\phi_uu^2dx\bigg)^{\frac{1}{2}}\\
&\leq\frac{\mu-4}{8}\int_{\mathbb{R}^3}\phi_uu^2dx+C\|u\|^2\|u-A_\varepsilon(u)\|^2
\end{split}
\end{equation}
and
\begin{equation}\label{eq3.10}
\begin{split}
\bigg|\int_{\mathbb{R}^3}\chi_\varepsilon(x)(u-A_\varepsilon(u))udx\bigg|&\leq \bigg(\int_{\mathbb{R}^3}\chi_\varepsilon(x)u^2dx\bigg)^{\frac{1}{2}}\bigg(\int_{\mathbb{R}^3}\chi_\varepsilon(x)(u-A_\varepsilon(u))^2dx\bigg)^{\frac{1}{2}}\\
&\leq\frac{\mu-2\beta}{4\beta}\int_{\mathbb{R}^3}\chi_\varepsilon(x)u^2dx+C\|u-A_\varepsilon(u)\|^2.
\end{split}
\end{equation}
It follows from \eqref{eq3.8}-\eqref{eq3.10} that
\begin{equation}\label{eq3.11}
\begin{split}
&\quad\frac{\min\{1,a_1\}(\mu -2)}{2\mu}\|u\|^2+\big(\int_{\mathbb{R}^3}\chi_\varepsilon(x)u^2dx-1\big)_+^\beta\\
&\leq C\|u\|^2\|u-A_\varepsilon(u)\|^2+\frac{\mu+2\beta}{2\mu}\big(\int_{\mathbb{R}^3}\chi_\varepsilon(x)u^2dx-1\big)_+^{\beta-1}\int_{\mathbb{R}^3}\chi_\varepsilon(x)u^2dx\\
&\quad+C\big(\int_{\mathbb{R}^3}\chi_\varepsilon(x)u^2dx-1\big)_+^{\beta-1}\|u-A_\varepsilon(u)\|^2+|\Phi_\varepsilon(u)|+C\|u\|\|u-A_\varepsilon(u)\|\\
&\leq C\|u\|^2\|u-A_\varepsilon(u)\|^2+\frac{\mu+2\beta}{2\mu}\big(\int_{\mathbb{R}^3}\chi_\varepsilon(x)u^2dx-1\big)_+^\beta+|\Phi_\varepsilon(u)|+C\|u\|\|u-A_\varepsilon(u)\|\\
&\quad+\big(\int_{\mathbb{R}^3}\chi_\varepsilon(x)u^2dx-1\big)_+^{\beta-1}(\frac{\mu+2\beta}{2\mu}+C\|u-A_\varepsilon(u)\|^2)\\
&\leq C\|u\|^2\|u-A_\varepsilon(u)\|^2+\frac{\mu+2\beta}{2\mu}\big(\int_{\mathbb{R}^3}\chi_\varepsilon(x)u^2dx-1\big)_+^\beta+|\Phi_\varepsilon(u)|+C\|u\|\|u-A_\varepsilon(u)\|\\
&\quad+\frac{\mu-2\beta}{4\mu}\big(\int_{\mathbb{R}^3}\chi_\varepsilon(x)u^2dx-1\big)_+^\beta+C(1+\|u-A_\varepsilon(u)\|^{2\beta})
\end{split}
\end{equation}
which implies that
\begin{equation}\label{eq3.12}
\begin{split}
&\quad\frac{\min\{1,a_1\}(\mu-2)}{2\mu}\|u\|^2\\
&\leq C\|u\|^2\|u-A_\varepsilon(u)\|^2+C(1+\|u-A_\varepsilon(u)\|^{2\beta})+|\Phi_\varepsilon(u)|+C\|u\|\|u-A_\varepsilon(u)\|.
\end{split}
\end{equation}
\par
If there exists $\{u_n\}\subset H^1(\mathbb{R}^3)$ with $|\Phi_\varepsilon(u_n)|\leq c_0$ and $\|\Phi_\varepsilon^\prime(u_n)\|\geq b_0$ such that $\|u_n-A_\varepsilon(u_n)\|\rightarrow 0$ as $n\rightarrow\infty$, then it follows from \eqref{eq3.12} that $\{u_n\}$ is bounded in $H^1(\mathbb{R}^3)$,
and by Lemma \ref{Lem3.4} we see that $\|\Phi_\varepsilon(u_n)\|\rightarrow 0$ as $n\rightarrow\infty$, which is a contradiction. Thus, \eqref{eq3.7} holds and the proof is completed.
\end{proof}
\par
Finally we consider the assumption $(\Gamma)$.
\begin{Lem}\label{Lem3.6}
$\Gamma_j$ is nonempty.
\end{Lem}
\begin{proof}
For any $n\in \mathbb{N}$, we choose $\{v_i\}_1^n\subset C_0^\infty(\mathbb{R}^3)\setminus \{0\}$ such that $\text{supp}(v_i)\cap \text{supp}(v_j)=\emptyset$ for $i\neq j$. Thus, for $\varepsilon$ small enough, one has
\begin{equation*}
B_0:=\{x\in \mathbb{R}^3|\cup_{i=1}^n\text{supp}(v_i)\}\subset \Lambda_\varepsilon.
\end{equation*}
Denote by $B_n$ the unit ball in $\mathbb{R}^n$, define $\varphi_n\in C(B_n,C_0^\infty(B_0))$ as
\begin{equation*}
\varphi_n(t)=R\sum_{i=1}^nt_iv_i,~t=(t_1,t_2,\ldots,t_n)\in B_n,
\end{equation*}
where $R>0$ is a large number. Obviously, $\varphi_n(0)=0\in P_\sigma^+\cap P_\sigma^-$ and $\varphi_n(-t)=-\varphi_n(t)$ for $t\in B_n$.
\par
Let
\begin{equation*}
\mathcal{J}(u)=\frac{1}{2}\int_{\mathbb{R}^3}|\nabla u|^2dx+\frac{a_2}{2}\int_{\mathbb{R}^3}u^2dx+\frac{1}{4}\int_{\mathbb{R}^3}\phi_uu^2dx-b_1\int_{\mathbb{R}^3}F(u)dx.
\end{equation*}
Thus, it follows from \eqref{eq1.10} and \eqref{eq2.3} that
\begin{equation*}
\begin{split}
\Phi_\varepsilon(\varphi_n(t))&\leq \mathcal{J}(\varphi_n(t))\leq\frac{1}{2}\int_{\mathbb{R}^3}|\nabla \varphi_n(t)|^2dx+(\frac{a_2}{2}+C_1)\int_{\mathbb{R}^3}\varphi_n^2(t)dx
\\
&\quad+C\big(\int_{\mathbb{R}^3}|\varphi_n(t)|^{\frac{12}{5}}dx\big)^{\frac{5}{3}}-C_0\int_{\mathbb{R}^3}|\varphi_n(t)|^\mu dx\\
&\leq CR^2\sum\limits_{i=1}^n\int_{\mathbb{R}^3}t_i^2(|\nabla v_i|^2+v_i^2\big)dx+CR^4\bigg(\sum\limits_{i=1}^n\int_{\mathbb{R}^3}t_i^{\frac{12}{5}}|v_i|^{\frac{12}{5}}dx\bigg)^{\frac{5}{3}}\\
&\quad-C_0R^\mu\sum\limits_{i=1}^n\int_{\mathbb{R}^3}t_i^\mu|v_i|^\mu dx.
\end{split}
\end{equation*}
Since $\mu>4$, one sees that $\Phi_\varepsilon(\varphi_n(t))\rightarrow -\infty$ as $R\rightarrow \infty$ uniformly for $t\in \partial B_n$. Hence, choosing $R$ large enough, we have
\begin{equation*}
\sup_{u\in \varphi_n(\partial B_n)}\Phi_\varepsilon(u)<c^*:=\inf_{u\in \Sigma}\Phi_\varepsilon(u).
\end{equation*}
Moreover, it is not difficult to check that $\varphi_n(\partial B_n)\cap (P_\sigma^+\cap P_\sigma^-)=\emptyset$ for $R>0$ large enough.
\par
Now we let
\begin{equation*}
\Theta=\{\eta|\eta \in C(H^1(\mathbb{R}^3),H^1(\mathbb{R}^3)),\eta~\text{odd},~\eta(P_\sigma^+)\subset P_\sigma^+,\eta(P_\sigma^-)\subset P_\sigma^-,\eta(u)=u~\text{if}~\Phi_\varepsilon(u)<0\},
\end{equation*}
\begin{equation*}
\Gamma_j=\{E|E\subset H^1(\mathbb{R}^3),E~\text{compact},-E=E,\gamma(E\cap \eta^{-1}(\Sigma))\geq j~\text{for}~\eta\in \Theta\},
\end{equation*}
then it follows from \cite[Lemma 4.2]{Liu-Liu-Wang2016JDE} that $\varphi_n(B_n)\subset \Gamma_{n-1}$. This completes this proof.
\end{proof}
\par
Having verified all the assumptions of Theorem \ref{Thm3.1}, we have the following existence theorem.
\begin{Thm}\label{Thm3.2}
Under the assumptions of Theorem \ref{Thm1.1}, the functional $\Phi_\varepsilon$ has infinitely many sign-changing critical points for $\varepsilon>0$ small,
\begin{equation*}
\{\pm u_{j,\varepsilon}|j=1,2,\cdots\}
\end{equation*}
and the corresponding critical values are defined as
\begin{equation*}
c_j^\varepsilon=\inf_{A\in \Gamma_j}\sup_{u\in A\setminus O}\Phi_\varepsilon(u),~j=1,2,\cdots.
\end{equation*}
Moreover,
\begin{itemize}
\item[$1.$] there exists $\tilde{c}_j,j=1,2,\cdots$, independent of $\varepsilon$ such that
\begin{equation}\label{eq3.13}
c_j^\varepsilon\leq \tilde{c}_j,~j=1,2,\cdots.
\end{equation}
\item[$2.$] If $c_j^\varepsilon=c_{j+\varepsilon}^\varepsilon=\cdots=c_{j+k}^\varepsilon=c$, then $\gamma(K_c^*)\geq k+1$.
\end{itemize}
\end{Thm}
\begin{proof}
It remains to verify \eqref{eq3.13}. Let $H_j:=\varphi_{j+1}(B_{j+1})\subset \Gamma_j$. For $t\in B_{j+1},u=\varphi_{j+1}(t)$, then $\big(\int_{\mathbb{R}^3}\chi_\varepsilon(x)u^2dx-1\big)_+^\beta=0$ and for $\varepsilon$ small enough
\begin{equation*}
\Phi_\varepsilon(u)\leq \mathcal{J}(u),~\forall u\in \varphi_{j+1}(B_{j+1}).
\end{equation*}
Hence,
\begin{equation*}
c_j^\varepsilon\leq \tilde{c}_j:=\sup_{H_j}\mathcal{J}(u).
\end{equation*}
This completes this proof.
\end{proof}

\section{Localization of nodal solutions and the proof of Theorem 1.1}
In this section, we are going to prove that the sign-changing critical points obtained in
Theorem \ref{Thm3.2} are solutions of the original system \eqref{eq2.1}.
\par
For any $k\in \mathbb{N}$, by Theorem \ref{Thm3.2}, there exists $\varepsilon_k^\prime>0$ such that, for any $\varepsilon\in (0,\varepsilon_k^\prime)$, the functional $\Phi_\varepsilon$ has at least $k$ pairs sign-changing critical points $\pm u_{j,\varepsilon}, j=1,\cdots,k$ and the corresponding critical values satisfy
\begin{equation*}
0<c_1^\varepsilon\leq c_2^\varepsilon\leq\cdots \leq c_k^\varepsilon\leq \tilde{c}_k.
\end{equation*}
Moreover, we have the following Lemma.
\begin{Lem}\label{Lem4.1}
There exist a positive constant $\rho$ depending
only on $a$ and $p$ and a positive constant $\eta_k$ independent of $\varepsilon$ such that
\begin{equation*}
\rho\leq \|u_{j,\varepsilon}\|\leq \eta_k~~\text{and}~~Q_\varepsilon(u_{j,\varepsilon})\leq \eta_k,~1\leq j \leq k.
\end{equation*}
\end{Lem}
\begin{proof}
By
\begin{equation*}
\begin{split}
\tilde{c}_k&\geq c_j^\varepsilon=\Phi_\varepsilon(u_{j,\varepsilon})-\frac{1}{\mu}\langle \Phi_\varepsilon^\prime(u_{j,\varepsilon}),u_{j,\varepsilon}\rangle\\
&=\frac{\mu-2}{2\mu}\bigg(\int_{\mathbb{R}^3}|\nabla u_{j,\varepsilon}|^2dx+\int_{\mathbb{R}^3}V(\varepsilon x)u_{j,\varepsilon}^2dx\bigg)+\frac{\mu-4}{4\mu}\int_{\mathbb{R}^3}\phi_{u_{j,\varepsilon}}u_{j,\varepsilon}^2dx+\big(\int_{\mathbb{R}^3}\chi_\varepsilon(x)u_{j,\varepsilon}^2dx-1\big)_+^\beta\\
&\quad-\frac{2\beta}{\mu}\bigg(\int_{\mathbb{R}^3}\chi_\varepsilon(x)u_{j,\varepsilon}^2dx-1\bigg)_+^{\beta-1}
\int_{\mathbb{R}^3}\chi_\varepsilon(x)u_{j,\varepsilon}^2dx+\int_{\mathbb{R}^3}K(\varepsilon x)[\frac{1}{\mu}f(u_{j,\varepsilon})u_{j,\varepsilon}-F(u_{j,\varepsilon})]dx\\
&\geq \frac{\min\{1,a_1\}(\mu-2)}{2\mu}\|u_{j,\varepsilon}\|^2+\big(\int_{\mathbb{R}^3}\chi_\varepsilon(x)u_{j,\varepsilon}^2dx-1\big)_+^\beta-\frac{2\beta}{\mu}\bigg(\int_{\mathbb{R}^3}\chi_\varepsilon(x)u_{j,\varepsilon}^2dx-1\bigg)_+^{\beta-1}
\int_{\mathbb{R}^3}\chi_\varepsilon(x)u_{j,\varepsilon}^2dx\\
&\geq \frac{\min\{1,a_1\}(\mu-2)}{2\mu}\|u_{j,\varepsilon}\|^2+\frac{\mu-2\beta}{\mu}\big(\int_{\mathbb{R}^3}\chi_\varepsilon(x)u_{j,\varepsilon}^2dx-1\big)_+^\beta
-\frac{2\beta}{\mu}\bigg(\int_{\mathbb{R}^3}\chi_\varepsilon(x)u_{j,\varepsilon}^2dx-1\bigg)_+^{\beta-1}
\end{split}
\end{equation*}
and $2<\beta<\frac{\mu}{2}$, we get that there exists $\eta_k>0$ independent of $\varepsilon$ such that $\|u_{j,\varepsilon}\|\leq \eta_k$ and $Q_\epsilon(u_{j,\epsilon})\leq \eta_k$.
\par
From $\langle \Phi_\varepsilon^\prime(u_{j,\epsilon}),u_{j,\varepsilon}\rangle=0$ and \eqref{eq1.9}, we get that
\begin{equation*}
\begin{split}
\min\{1,a_1\}\|u_{j,\varepsilon}\|^2&\leq\int_{\mathbb{R}^3}|\nabla u_{j,\varepsilon}|^2dx+\int_{\mathbb{R}^3}V(\varepsilon x)u_{j,\varepsilon}^2dx+\int_{\mathbb{R}^3}\phi_{u_{j,\varepsilon}}u_{j,\varepsilon}^2dx\\
&\quad+2\beta\big(\int_{\mathbb{R}^3}\chi_\varepsilon(x)u_{j,\varepsilon}^2dx-1\big)_+^{\beta-1}\int_{\mathbb{R}^3}\chi_\varepsilon(x)u_{j,\varepsilon}^2dx\\
&=\int_{\mathbb{R}^3}K(\varepsilon x)f(u_{j,\varepsilon})u_{j,\varepsilon}dx\\
&\leq \frac{\min\{1,a_1\}}{2}\|u_{j,\varepsilon}\|^2+C\|u_{j,\varepsilon}\|^p,
\end{split}
\end{equation*}
which implies that
\begin{equation*}
\frac{\min\{1,a_1\}}{2}\|u_{j,\varepsilon}\|^2\leq C\|u_{j,\varepsilon}\|^p.
\end{equation*}
Since $p>2$ and $u_{j,\epsilon}\neq 0$, we deduce that there exists $\rho>0$ depending only on $a_1$ and $p$ such that $\|u_{j,\varepsilon}\|\geq \rho,1\leq j \leq k$.
\end{proof}

\begin{Lem}\label{Lem4.2}
Assume $\Phi_\varepsilon^\prime(u)=0,\Phi_\varepsilon(u)\leq L$. Then there exists $c=c(L)$ such that $|u(x)|\leq c$ for $x\in \mathbb{R}^3$. Moreover, for any $\delta>0$ there exists $c=c(\delta, L)$ such that $|u(x)|\leq c\varepsilon^3$ for $x\in \mathbb{R}^3\setminus (\Lambda_\varepsilon)^\delta$.
\end{Lem}
\begin{proof}
The proof is an application of Moser's iteration.\\
$(1)$ Define
\begin{equation*}
u_T(x)=
\begin{cases}
-T,&\text{if}~u(x)\leq -T,\\
u(x),&\text{if}~|u(x)|\leq T,\\
T,&\text{if}~u(x)\geq T,
\end{cases}
\end{equation*}
where $T>0$. Set $\varphi=|u_T|^{2k-2}u$ with $k\geq 1$ as test function in $\langle \Phi_\varepsilon^\prime(u),\varphi\rangle=0$, that is
\begin{equation}\label{eq4.1}
\int_{\mathbb{R}^3}\nabla u \nabla \varphi dx+\int_{\mathbb{R}^3}V(\varepsilon x) u\varphi dx+\int_{\mathbb{R}^3}\phi_u u  \varphi dx
+2\beta \bigg(\int_{\mathbb{R}^3}\chi_\varepsilon(x)u^2dx-1\bigg)_+^{\beta-1}\int_{\mathbb{R}^3}\chi_\varepsilon(x) u  \varphi dx=\int_{\mathbb{R}^3}K(\varepsilon x)f(u)\varphi dx.
\end{equation}
Note that
\begin{equation*}
\int_{\mathbb{R}^3}\phi_u u  \varphi dx
+2\beta \bigg(\int_{\mathbb{R}^3}\chi_\varepsilon(x)u^2dx-1\bigg)_+^{\beta-1}\int_{\mathbb{R}^3}\chi_\varepsilon(x)u\varphi dx\geq 0.
\end{equation*}
Thus, by \eqref{eq4.1} and \eqref{eq1.9}, one has
\begin{equation}\label{eq4.2}
\int_{\mathbb{R}^3}|u_T|^{2k-2}| \nabla u|^2dx+2(k-1)\int_{\{|u(x)|\leq T\}}|u|^{2k-2}|\nabla u|^2dx\leq C\int_{\mathbb{R}^3}|u_T|^{2k-2}|u|^pdx.
\end{equation}
Moreover, by H\"older inequality, we obtain
\begin{equation}\label{eq4.3}
\int_{\mathbb{R}^3}|u_T|^{2k-2}|u|^pdx=\int_{\mathbb{R}^3}(u|u_T|^{k-1})^2|u|^{p-2}dx\leq\big(\int_{\mathbb{R}^3}|u|^{2^*}dx\big)^{\frac{p-2}{2^*}} \big(\int_{\mathbb{R}^3}(u|u_T|^{k-1})^{\frac{2\cdot 2^*}{2^*-p+2}}dx\big)^{\frac{2^*-p+2}{2^*}}.
\end{equation}
A direct estimation and \eqref{eq4.2}-\eqref{eq4.3} imply that
\begin{equation}\label{eq4.4}
\begin{split}
\int_{\mathbb{R}^3}|\nabla (u|u_T|^{k-1})|^2dx&=\int_{\mathbb{R}^3}|\nabla u|^2|u_T|^{2k-2}dx+
(k^2-1)\int_{\{|u(x)|\leq T\}}|u|^{2k-2}|\nabla u|^2dx\\
&\leq \frac{k+1}{2}\int_{\mathbb{R}^3}|\nabla u|^2|u_T|^{2k-2}dx+
(k^2-1)\int_{\{|u(x)|\leq T\}}|u|^{2k-2}|\nabla u|^2dx\\
&=\frac{k+1}{2}\bigg(\int_{\mathbb{R}^3}|\nabla u|^2|u_T|^{2k-2}dx+
2(k-1)\int_{\{|u(x)|\leq T\}}|u|^{2k-2}|\nabla u|^2dx\bigg)\\
&\leq Ck\int_{\mathbb{R}^3}|u_T|^{2k-2}|u|^pdx\\
&\leq Ck\big(\int_{\mathbb{R}^3}(u|u_T|^{k-1})^{2^*\cdot\frac{2\cdot}{2^*-p+2}}dx\big)^{\frac{2^*-p+2}{2^*}}.
\end{split}
\end{equation}
Thus, by Sobolev inequality and \eqref{eq4.4}, we have
\begin{equation}\label{eq4.5}
\big(\int_{\mathbb{R}^3}(u|u_T|^{k-1})^{2^*}dx\big)^{\frac{2}{2^*}}\leq Ck\big(\int_{\mathbb{R}^3}(u|u_T|^{k-1})^{\frac{2\cdot 2^* }{2^*-p+2}}dx\big)^{\frac{2^*-p+2}{2^*}}.
\end{equation}
Assume $\int_{\mathbb{R}^3}|u|^{k\cdot{\frac{2\cdot 2^*}{2^*-p+2}}}dx<\infty$. Let $T\rightarrow \infty$ in \eqref{eq4.5}, we obtain
\begin{equation*}
\big(\int_{\mathbb{R}^3}|u|^{k\cdot2^*}dx\big)^{\frac{2}{2^*}}\leq Ck\big(\int_{\mathbb{R}^3}|u|^{k\cdot \frac{2\cdot 2^* }{2^*-p+2}}dx\big)^{\frac{2^*-p+2}{2^*}}.
\end{equation*}
Denote $\chi=\frac{2^*-p+2}{2}>1$. Starting from $k_1=\chi$, then
\begin{equation*}
\big(\int_{\mathbb{R}^3}|u|^{\chi\cdot2^*}dx\big)^{\frac{1}{\chi2^*}}\leq C\chi^{\frac{1}{\chi}}\big(\int_{\mathbb{R}^3}|u|^{2^*}dx\big)^{\frac{1}{2^*}}.
\end{equation*}
By iteration, we have
\begin{equation*}
\|u\|_{\chi^n\cdot 2^*}\leq C(\chi^n)^{\frac{1}{\chi^n}}\|u\|_{2^*},n=1,2,\cdots.
\end{equation*}
Hence,
\begin{equation*}
\|u\|_\infty\leq C\|u\|_{2^*}\leq C.
\end{equation*}
$(2)$ For $x_0\in \mathbb{R}^3,0<\rho<R\leq 1$. Choose $\eta\in C_0^\infty(\mathbb{R}^3,[0,1])$ such that $\eta(x)=1$ for $x\in B_\rho=B_\rho(x_0)$; $\eta(x)=0$ for $x\notin B_R(x_0)$ and $|\nabla \eta|\leq \frac{c}{R-\rho}$. Set $\varphi=u|u|^{2k-2}\eta^2,k\geq 1$ as test function in $\langle \Phi_\varepsilon^\prime(u),\varphi \rangle=0$ and similar to \eqref{eq4.2}, we have
\begin{equation}\label{eq4.6}
(2k-1)\int_{\mathbb{R}^3}|u|^{2k-2}|\nabla u|^2\eta^2dx+2\int_{\mathbb{R}^3}u|u|^{2k-2}\eta \nabla u \nabla \eta dx\leq
C\int_{\mathbb{R}^3}|u|^{2k+p-2}\eta^2dx.
\end{equation}
Note that, by the $L^\infty$-estimate of $u$,
\begin{equation}\label{eq4.7}
\int_{\mathbb{R}^3}|u|^{2k+p-2}\eta^2dx\leq C\int_{B_R(x_0)}|u|^{2k}dx.
\end{equation}
Thus, by \eqref{eq4.6} and \eqref{eq4.7}
\begin{equation*}
\begin{split}
\int_{\mathbb{R}^3}|\nabla (|u|^k\eta)|^2dx&=\int_{\mathbb{R}^3}|\nabla \eta|^2|u|^{2k}dx+
k^2\int_{\mathbb{R}^3}|u|^{2k-2}|\nabla u|^2\eta^2dx+2k\int_{\mathbb{R}^3}|u|^{2k-2}u \eta \nabla u \nabla \eta dx\\
&\leq \int_{\mathbb{R}^3}|\nabla \eta|^2|u|^{2k}dx+
k\bigg((2k-1)\int_{\mathbb{R}^3}|u|^{2k-2}|\nabla u|^2\eta^2dx+2\int_{\mathbb{R}^3}|u|^{2k-2}u \eta \nabla u \nabla \eta dx\bigg)\\
&\leq \int_{\mathbb{R}^3}|\nabla \eta|^2|u|^{2k}dx+
Ck\int_{\mathbb{R}^3}|u|^{2k+p-2}\eta^2dx\\
&\leq \frac{C}{(R-\rho)^2}\int_{B_R(x_0)}|u|^{2k}dx+Ck\int_{B_R(x_0)}|u|^{2k}dx\\
&\leq \frac{Ck}{(R-\rho)^2}\int_{B_R(x_0)}|u|^{2k}dx,
\end{split}
\end{equation*}
which implies that
\begin{equation*}
\big(\int_{B_\rho(x_0)}|u|^{k\cdot 2^*}dx\big)^{\frac{2}{2^*}}\leq \frac{Ck}{(R-\rho)^2}\int_{B_R(x_0)}|u|^{2k}dx.
\end{equation*}
By iteration we have
\begin{equation*}
\|u\|_{L^\infty(B_{\frac{R}{2}}(x_0))}\leq \|u\|_{L^2(B_R(x_0))}.
\end{equation*}
We claim that
\begin{equation}\label{eq4.8}
\int_{\mathbb{R}^3\setminus (\Lambda_\varepsilon)^\delta}u^2dx\leq C_\delta \varepsilon^6.
\end{equation}
In fact, $\int_{\mathbb{R}^3}\chi_\varepsilon u^2dx
\leq (\int_{\mathbb{R}^3}\chi_\varepsilon u^2dx-1)_++1\leq C$. By the definition, if $\delta\geq 1$
\begin{equation*}
\int_{\mathbb{R}^3\setminus (\Lambda_\varepsilon)^\delta}u^2dx=\varepsilon^6\int_{\mathbb{R}^3\setminus (\Lambda_\varepsilon)^\delta}\chi_\varepsilon(x) u^2dx\leq
\varepsilon^6\int_{\mathbb{R}^3}\chi_\varepsilon(x) u^2dx\leq C\varepsilon^6.
\end{equation*}
If $0<\delta<1$, then
\begin{equation*}
\begin{split}
\int_{\mathbb{R}^3}\chi_\varepsilon(x) u^2dx&\geq\int_{\mathbb{R}^3\setminus(\Lambda_\varepsilon)^\delta}\chi_\varepsilon(x) u^2dx\\
&=\varepsilon^{-6}\int_{\mathbb{R}^3\setminus(\Lambda_\varepsilon)^1} u^2dx+\varepsilon^{-6}\int_{(\Lambda_\varepsilon)^1\setminus(\Lambda_\varepsilon)^\delta}\zeta(\text{dist}(x,\Lambda_\varepsilon))u^2dx\\
&\geq\varepsilon^{-6}\bigg(\int_{\mathbb{R}^3\setminus(\Lambda_\varepsilon)^1} u^2dx+\min_{\tau\in[\delta,1]}\zeta(\tau)
\int_{(\Lambda_\varepsilon)^1\setminus(\Lambda_\varepsilon)^\delta}u^2dx\bigg)\\
&\geq \min\{1,\min_{\tau\in[\delta,1]}\zeta(\tau)\}\varepsilon^{-6}\int_{\mathbb{R}^3\setminus(\Lambda_\varepsilon)^\delta} u^2dx
\end{split}
\end{equation*}
where
\begin{equation*}
\min_{\tau\in[\delta,1]}\zeta(\tau)>0.
\end{equation*}
Thus,
\begin{equation*}
\int_{\mathbb{R}^3\setminus (\Lambda_\varepsilon)^\delta}u^2dx=\varepsilon^6\int_{\mathbb{R}^3\setminus (\Lambda_\varepsilon)^\delta}\chi_\varepsilon(x) u^2dx\leq
\varepsilon^6\int_{\mathbb{R}^3}\chi_\varepsilon(x) u^2dx\leq C\varepsilon^6.
\end{equation*}
So, claim \eqref{eq4.8} holds and then
\begin{equation*}
|u(x)|\leq C_\delta \varepsilon^6~\text{for}~x\in \mathbb{R}^3\setminus (\Lambda_\varepsilon)^\delta.
\end{equation*}
This completes this proof.
\end{proof}

For fixed any $1\leq j\leq k$ and let $\varepsilon_n\rightarrow 0^+$ as $n\rightarrow\infty$, by Lemma \ref{Lem4.1}, $\{u_{j,\varepsilon_n}\}$ is bounded in $H^1(\mathbb{R}^3)$. So, we can use the following profile decomposition result introduced in \cite{Tintarev-Fieseler2007book}.
\begin{Lem}\label{Lem4.3}
For fixed any $1\leq j\leq k$ and let $\varepsilon_n\rightarrow 0^+$ as $n\rightarrow\infty$, there exist $U_j^i,~r_{j,\varepsilon_n}\in H^1(\mathbb{R}^3), y_{j,\varepsilon_n}^i\in \mathbb{R}^3$ such that
\begin{equation}\label{eq4.9}
u_{j,\varepsilon_n}=\sum_{i}U_j^i(\cdot-y_{j,\varepsilon_n}^i)+r_{j,\varepsilon_n}
\end{equation}
and satisfy
\begin{itemize}
\item[$(1)$] $u_{j,\varepsilon_n}(\cdot+y_{j,\varepsilon_n}^i)\rightharpoonup U_j^i$ in $H^1(\mathbb{R}^3)$ as $n\rightarrow \infty$.
\item[$(2)$] $|y_{i,\varepsilon_n}^i-y_{j,\varepsilon_n}^{i^\prime}|\rightarrow\infty$ as $n\rightarrow\infty$ for $i\neq i^\prime$.
\item[$(3)$] $\|u_{j,\varepsilon_n}\|^2=\sum\limits_{i}\|U_j^i\|^2+\|r_{j,\varepsilon_n}\|^2+o_n(1)$.
\item[$(4)$] $\|r_{j,\varepsilon_n}\|_s=o_n(1)$ and $\|u_{j,\varepsilon_n}\|_s^s=\sum\limits_{i}\|U_j^i\|_s^s+o_n(1), s\in (2,6)$.
\end{itemize}
\end{Lem}
By Lemma \ref{Lem4.2}(2),
\begin{equation*}
\lim_{n\rightarrow\infty}\text{dist}(y_{j,\varepsilon_n}^i,\Lambda_{\varepsilon_n})<\infty.
\end{equation*}
Up to a subsequence, we assume that
\begin{equation}\label{eq4.10}
y_j^i=\lim_{n\rightarrow\infty}\varepsilon_n y_{j,\varepsilon_n}^i.
\end{equation}
Since $\text{dist}(y_{j,\varepsilon_n}^i,\Lambda_{\varepsilon_n})=\varepsilon_n^{-1}\text{dist}(\varepsilon_n y_{j,\varepsilon_n}^i,\Lambda)$, we have
\begin{equation}\label{eq4.11}
\text{dist}(y_j^i,\Lambda)=0,~i.e.,~y_j^i\in \overline{\Lambda}.
\end{equation}

A similar argument of \cite[Lemma 3.2]{Liu-Liu-Wang2019JDE} but easier, we have the following Lemma.
\begin{Lem}\label{Lem4.4}
Assume $\varepsilon_n\rightarrow0^+$ as $n\rightarrow\infty$, $\Phi_{\varepsilon_n}^\prime(u_n)=0, \{u_n\}$ is bounded in $H^1(\mathbb{R}^3)$. Assume $\tilde{u}_n=u_n(\cdot +y_n)\rightharpoonup U$ in $H^1(\mathbb{R}^3), y_n\in \mathbb{R}^3, \lim\limits_{n\rightarrow\infty} \varepsilon_n y_n=y^*$.
\begin{itemize}
\item[$(1)$] If $\lim\limits_{n\rightarrow \infty}\text{dist}(y_n,\partial \Lambda_{\varepsilon_n})=\infty$, then $y_n\in \Lambda_{\varepsilon_n}$ and $U$ satisfies
\begin{equation*}
\int_{\mathbb{R}^3}\nabla U \nabla \varphi dx+\int_{\mathbb{R}^3}V(y^*)U\varphi dx+\int_{\mathbb{R}^3}\phi_UU\varphi dx=\int_{\mathbb{R}^3}K(y^*)f(U)\varphi dx~~\text{for}~\varphi\in C_0^\infty(\mathbb{R}^3).
\end{equation*}
\item[$(2)$] If $\lim\limits_{n\rightarrow \infty}\text{dist}(y_n,\partial \Lambda_{\varepsilon_n})<\infty$ (without loss of generality we assume $\lim\limits_{n\rightarrow \infty}\text{dist}(y_n,\partial \Lambda_{\varepsilon_n})=0$), then $U$ satisfies
\begin{equation*}
\int_{\mathbb{R}_+^3}\nabla U \nabla \varphi dx+\int_{\mathbb{R}_+^3}V(y^*)U\varphi dx+\int_{\mathbb{R}_+^3}\phi_UU\varphi dx=\int_{\mathbb{R}_+^3}K(y^*)f(U)\varphi dx~~\text{for}~\varphi\in C_0^\infty(\mathbb{R}_+^3)
\end{equation*}
and $U(x)=0$ for $x=(x_1,x_2,x_3),x_3\leq 0$, where $\mathbb{R}_+^3=\{x| x=(x_1,x_2,x_3)\in \mathbb{R}^3,x_3>0\}$.
\end{itemize}
\end{Lem}

\begin{Lem}\label{Lem4.5}
The summation in the profile decomposition \eqref{eq4.9} has only finitely many terms.
\end{Lem}
\begin{proof}
In both cases of Lemma \ref{Lem4.4}, we have
\begin{equation*}
\begin{split}
\int_{\mathbb{R}^3}|\nabla U_j^i|^2dx+a_1\int_{\mathbb{R}^3}|U_j^i|^2dx&\leq\int_{\mathbb{R}^3}|\nabla U_j^i|^2dx+\int_{\mathbb{R}^3}V(y_j^i)|U_j^i|^2dx+
\int_{\mathbb{R}^3}\phi_{U_j^i}|U_j^i|^2dx\\
&=\int_{\mathbb{R}^3}K(y_j^i)f(U_j^i)U_j^idx\\
&\leq \frac{a_1}{2}\int_{\mathbb{R}^3}|U_j^i|^2dx+C\int_{\mathbb{R}^3}|U_j^i|^pdx
\end{split}
\end{equation*}
which implies that
\begin{equation*}
\int_{\mathbb{R}^3}|\nabla U_j^i|^2dx+\frac{a_1}{2}\int_{\mathbb{R}^3}|U_j^i|^2dx\leq C\int_{\mathbb{R}^3}|U_j^i|^pdx.
\end{equation*}
By H\"older inequality and Sobolev imbedding inequality
\begin{equation*}
\begin{split}
\int_{\mathbb{R}^3}|U_j^i|^pdx&\leq\big(\int_{\mathbb{R}^3}|U_j^i|^2dx\big)^t\big(\int_{\mathbb{R}^3}|U_j^i|^6dx\big)^{1-t}\\
&\leq C\big(\int_{\mathbb{R}^3}|U_j^i|^2dx\big)^t\big(\int_{\mathbb{R}^3}| \nabla U_j^i|^2dx\big)^{3(1-t)}\\
&\leq C\big(\int_{\mathbb{R}^3}|U_j^i|^pdx\big)^{3-2t}
\end{split}
\end{equation*}
where $t=\frac{6-p}{4}\in (0,1)$. Thus, there exists $\tilde{c}_0>0$ such that $\int_{\mathbb{R}^3}|U_j^i|^pdx\geq \tilde{c}_0$. By the property $(4)$ of the profile decomposition \eqref{eq4.9} the summation has only finite terms.
\end{proof}

For fix any $1\leq j \leq k$, it follows from Lemma \ref{Lem4.4} and the property $(4)$ of the profile decomposition \eqref{eq4.9} that there exist $m_j$ nonzero functions $U_j^i$ in $H^1(\mathbb{R}^3), 1\leq i \leq m_j$ and satisfies \eqref{eq4.10}-\eqref{eq4.11}.
 We may write the set of these limiting points by
\begin{equation*}
\{y_j^1, y_j^2, \cdots,y_j^{s_j}\}=\{\lim_{n\rightarrow \infty}\varepsilon_n y_{j,\varepsilon_n}^i|0\leq i\leq m_j\}\subset \overline{\Lambda},
\end{equation*}
such that $1\leq s_j\leq m_j$ and $y_j^i\neq y_j^{i^\prime}$ for $1\leq i\neq i^\prime \leq s_j$. Set
\begin{equation}\label{eq4.12}
\vartheta_j=
\begin{cases}
\frac{1}{10}\min\{|y_j^i-y_j^{i^\prime}||1\leq i\neq i^\prime \leq s_j\},&\text{if}~s_j\geq 2,\\
\infty,&\text{if}~s_j=1.
\end{cases}
\end{equation}

\begin{Lem}\label{Lem4.6}
If
\begin{equation*}
0<\delta<\vartheta_j,
\end{equation*}
then there exist $C > 0$ and $c > 0$ independent of $n$ such that, for every $1 \leq i \leq m_j$, when $n$ is large enough,
\begin{equation*}
|\nabla u_{j,\varepsilon_n}(x)|+|u_{j,\varepsilon_n}(x)|\leq C\text{exp}(-c \varepsilon_n^{-1}),~~\forall x\in \overline{B(y_{j,\varepsilon_n}^i,\delta \varepsilon_n^{-1}+1)}\setminus B(y_{j,\varepsilon_n}^i,\delta \varepsilon_n^{-1}-1).
\end{equation*}
\end{Lem}
\begin{proof}
Its proof follows the argument as in \cite{Chen-Wang2017CVPDE}, but we present it here for the sake of completeness. We define
\begin{equation*}
A_n^i=\overline{B(y_{j,\varepsilon_n}^i,\frac{3}{2}\delta \varepsilon_n^{-1})}\setminus B(y_{j,\varepsilon_n}^i,\frac{1}{2}\delta \varepsilon_n^{-1}).
\end{equation*}
Then the definition of $\vartheta_j$ and the fact $0<\delta<\vartheta_j$, we deduce that, for every $1\leq i^\prime ,i\leq m_j$,
\begin{equation}\label{eq4.13}
\text{dist}(y_{j,\varepsilon_n}^{i^\prime},A_n^i)\to \infty,~\text{as}~n\rightarrow\infty.
\end{equation}
This, together with property $(4)$ of the profile decomposition \eqref{eq4.9} and the fact
\begin{equation}\label{eq4.14}
\lim_{R\to \infty}\int_{\mathbb{R}^3\setminus B(y_{j,\varepsilon_n}^i,R)}|U_j^i(\cdot-y_{j,\varepsilon_n}^i)|^pdx=0,~1\leq i\leq m_j,
\end{equation}
we get that, for every $1\leq i\leq m_j$,
\begin{equation*}
\lim_{n\to \infty}\int_{A_n^i}|u_{j,\varepsilon_n}|^pdx=0.
\end{equation*}
It follows that there exists $n_0\in \mathbb{N}$ such that for $n \geq n_0$,
\begin{equation}\label{eq4.15}
C(\frac{a_1}{2})|u_{j,{\varepsilon_n}}(x)|^{p-2}<\frac{a_1}{4},~\text{for~any}~x\in A_n^i, 1\leq i\leq m_j,
\end{equation}
where $C(\frac{a_1}{2})$ is defined in \eqref{eq1.9}. For non-negative integer $m$, let
\begin{equation*}
R_m=\overline{B(y_{j,\varepsilon_n}^i,\frac{3}{2}\delta \varepsilon_n^{-1}-m)}\setminus B(y_{j,\varepsilon_n}^i,\frac{1}{2}\delta \varepsilon_n^{-1}+m)
\end{equation*}
and let $\varsigma_m$ be a cut-off function satisfying that $0\leq \varsigma_m(t)\leq 1,|\varsigma_m^\prime(t)|\leq 4$ for all $t\in \mathbb{R}$ and
\begin{equation*}
\varsigma_m(t)=
\begin{cases}
0,&t\leq \frac{1}{2}\delta \varepsilon_n^{-1}+m-1~\text{or}~t\geq \frac{3}{2}\delta \varepsilon_n^{-1}-m+1,\\
1,&\frac{1}{2}\delta \varepsilon_n^{-1}+m\leq t \leq\frac{3}{2}\delta \varepsilon_n^{-1}-m.
\end{cases}
\end{equation*}
For $x\in \mathbb{R}^3$, let $\psi_m(x)=\varsigma_m(|x-y_{j,\varepsilon_n}^i|)$. Multiplying both sides of \eqref{eq2.5} by $\psi_m^2u_{j,\varepsilon_n}$ and integrating on $\mathbb{R}^3$, we get that
\begin{equation}\label{eq4.16}
\begin{split}
&\int_{R_{m-1}}|\nabla u_{j,\varepsilon_n}|^2\psi_m^2dx +\int_{R_{m-1}}\phi_{u_{j,\varepsilon_n}}u_{j,\varepsilon_n}^2\psi_m^2dx+\int_{R_{m-1}}V(\varepsilon_n x)u_{j,\varepsilon_n}^2\psi_m^2dx\\
&\quad+\xi_n\int_{R_{m-1}}\chi_{\varepsilon_n}(x)u_{j,\varepsilon_n}^2\psi_m^2dx-\int_{R_{m-1}}f(u_{j,\varepsilon_n})\psi_m^2u_{j,\varepsilon_n}dx\\
&=-2\int_{R_{m-1}}u_{j,\varepsilon_n}\psi_m\nabla u_{j,\varepsilon_n}\nabla \psi_mdx\\
&\leq 8\int_{R_{m-1}\setminus R_m}|u_{j,\varepsilon_n}|\cdot |\nabla u_{j,\varepsilon_n}|dx,
\end{split}
\end{equation}
where
\begin{equation}\label{eq4.17}
\xi_n:=2\beta\bigg(\int_{\mathbb{R}^3}\chi_{\varepsilon_n}(x)u_{j,\varepsilon_n}^2dx-1\bigg)_+^{\beta-1}.
\end{equation}
By \eqref{eq1.9} and \eqref{eq4.15}, we get that
\begin{equation}\label{eq4.18}
\begin{split}
&\int_{R_{m-1}}|\nabla u_{j,\varepsilon_n}|^2\psi_m^2dx +\int_{R_{m-1}}\phi_{u_{j,\varepsilon_n}}u_{j,\varepsilon_n}^2\psi_m^2dx+\int_{R_{m-1}}V(\varepsilon_n x)u_{j,\varepsilon_n}^2\psi_m^2dx\\
&\quad+\xi_n\int_{R_{m-1}}\chi_{\varepsilon_n}(x)u_{j,\varepsilon_n}^2\psi_m^2dx-\int_{R_{m-1}}f(u_{j,\varepsilon_n})\psi_m^2u_{j,\varepsilon_n}dx\\
&\geq \min\{1,\frac{a_1}{4}\}\int_{R_m}(|\nabla u_{j,\varepsilon_n}|^2+u_{j,\varepsilon_n}^2)dx.
\end{split}
\end{equation}
Combining \eqref{eq4.16} and \eqref{eq4.18} yields that
\begin{equation}\label{eq4.19}
\begin{split}
\int_{R_m}(|\nabla u_{j,\varepsilon_n}|^2+u_{j,\varepsilon_n}^2)dx&\leq \frac{8}{\min\{1,\frac{a_1}{4}\}}\int_{R_{m-1}\setminus R_m}|u_{j,\varepsilon_n}|\cdot |\nabla u_{j,\varepsilon_n}|dx\\
&\leq \frac{4}{\min\{1,\frac{a_1}{4}\}}\int_{R_{m-1}\setminus R_m}(|\nabla u_{j,\varepsilon_n}|^2+u_{j,\varepsilon_n}^2)dx.
\end{split}
\end{equation}
Setting $\alpha_m=\int_{R_m}(|\nabla u_{j,\varepsilon_n}|^2+u_{j,\varepsilon_n}^2)dx$ and $\tilde{C}=\frac{4}{\min\{1,\frac{a_1}{4}\}}$. Then, there holds $\alpha_m\leq \tilde{C}(\alpha_{m-1}-\alpha_m)$ which gives $\alpha_m\leq \theta \alpha_{m-1}$ with $\theta=\frac{\tilde{C}}{1+\tilde{C}}<1$. Hence $\alpha_m\leq \alpha_0\theta^m$. By Lemma \ref{Lem4.1}, we have $\alpha_0\leq \eta_k^2$. Thus, for sufficiently large $n, \alpha_m\leq \eta_k^2\theta^m=\eta_k^2e^{m \ln \theta}$. Let $[x]$ denote the
integer part of $x$. Choosing $m = [\frac{\delta\varepsilon_n^{-1}}{2}]-1$ and noting that $[\frac{\delta \varepsilon_n^{-1}}{2}]-1\geq  \frac{\delta\varepsilon_n^{-1}}{4}$ when $n$ is large enough, we get that
\begin{equation}\label{eq4.20}
\int_{D_n^i}(|\nabla u_{j,\varepsilon_n}|^2+u_{j,\varepsilon_n}^2)dx\leq \alpha_m\leq \eta_k^2\text{exp}(([\delta\varepsilon_n^{-1}/2]-1)\ln \theta)\leq \eta_k^2\text{exp}(\frac{1}{4}\delta \varepsilon_n^{-1}\ln \theta),
\end{equation}
where
\begin{equation*}
D_n^i=\overline{B(y_{j,\varepsilon_n}^i,\delta \varepsilon_n^{-1}+1)}\setminus B(y_{j,\varepsilon_n}^i,\delta \varepsilon_n^{-1}-1).
\end{equation*}
Then the result of this lemma follows from \eqref{eq4.20} and the standard regularity theory of elliptic equation (see \cite{Gilbarg-Trudinger1983book}).
\end{proof}
\par
Note that, from \eqref{eq1.12} in $(VK1)$, we deduce that there exists $\delta_0>0$ with $0< \delta_0<v_j$ such that
 \begin{equation}\label{eq4.21}
\sup_{x\in \Lambda^{\delta_0} \setminus U(\delta_0)}\nabla K(x)\cdot \nabla V(x)<0,
\end{equation}
where $v_j$ be the positive constant given in \eqref{eq4.12}.
\begin{Lem}\label{Lem4.7}
For every $1\leq i\leq m_j , \lim\limits_{\varepsilon \rightarrow 0}\text{dist}(\varepsilon y_{j,\varepsilon}^i,U(\delta_0)) = 0$.
\end{Lem}
\begin{proof}
Arguing indirectly, we assume that there exist $1\leq i_0\leq m_j$ and sequence $\varepsilon_n>0$ such that $\lim\limits_{n\to \infty}\varepsilon_n = 0$ and
 \begin{equation*}
\lim_{n\rightarrow\infty}\text{dist}(\varepsilon_n y_{j,\varepsilon_n}^{i_0},U(\delta_0))>0.
\end{equation*}
Note that, from \eqref{eq1.11}, we deduce that there exists $\delta_1>0$ such that, for any $y\in \Lambda^{\delta_1}$
\begin{equation}\label{eq4.22}
\inf_{x\in B(y,\delta_1)\setminus \Lambda}\nabla V(y)\cdot \nabla \text{dist}(x, \partial \Lambda)>0.
\end{equation}
For every $i$, without loss of generality, we may assume that $\lim\limits_{n\rightarrow\infty}\varepsilon_n y_{j,\varepsilon_n}^i$ exists.
It follows from $y_j^{i_0}= \lim\limits_{n\rightarrow\infty}\varepsilon_n y_{j,\varepsilon_n}^{i_0}\notin U(\delta_0)$ that
 \begin{equation*}
\nabla K(y_j^{i_0})\cdot\nabla V(y_j^{i_0})<0,
\end{equation*}
and then
 \begin{equation*}
\nabla V(y_j^{i_0})\neq 0.
\end{equation*}
Thus, we deduce that there exists $\delta_2\in (0,\delta_1)$, we may assume $\delta_2=\delta_0$(we choose $\delta_0$ small enough if necessary), such that for sufficiently large $n$,
 \begin{equation}\label{eq4.23}
\inf_{x\in B(y_{j,\varepsilon_n}^{i_0},\delta_0\varepsilon_n^{-1})}\nabla V(\varepsilon_n x)\cdot \nabla V(\varepsilon_ny_{j,\varepsilon_n}^{i_0})\geq\frac{1}{2}|\nabla V(y_j^{i_0})|^2>0,
\end{equation}
and
 \begin{equation}\label{eq4.24}
\sup_{x\in B(y_{j,\varepsilon_n}^{i_0},\delta_0\varepsilon_n^{-1})}\nabla K(\varepsilon_n x)\cdot \nabla V(\varepsilon_ny_{j,\varepsilon_n}^{i_0})<0.
\end{equation}
Then according to the definition of $\delta_0$ and Lemma \ref{Lem4.6}, there exist $C>0$ and $c>0$ independent of $n$ such that, for $n$ sufficiently large
\begin{equation}\label{eq4.25}
|\nabla u_{j,\varepsilon_n}(x)|+|u_{j,\varepsilon_n}(x)|\leq C\text{exp}(-c \varepsilon_n^{-1}),~~\forall x\in \overline{B(y_{j,\varepsilon_n}^{i_0},\delta_0 \varepsilon_n^{-1}+1)}\setminus B(y_{j,\varepsilon_n}^{i_0},\delta_0 \varepsilon_n^{-1}-1).
\end{equation}
By Lemma \ref{Lem4.1}, we deduce that there exists $C>0$ independent of $n$ such that $\xi_n$, defined by \eqref{eq4.17}, satisfies
\begin{equation}\label{eq4.26}
0\leq \xi_n\leq C,~\forall n.
\end{equation}
Let
\begin{equation*}
\vec{t}_n=\nabla V(\varepsilon_n y_{j,\varepsilon_n}^{i_0}).
\end{equation*}
Since $u_{j,\varepsilon_n}$ solves \eqref{eq2.5}, the elliptic regularity theory implies that $u_{j,\varepsilon_n}$ and $\phi=\phi_{u_{j,\varepsilon_n}}$ are,  at least, $C^1$ function. Multiplying both sides of \eqref{eq2.5} by $\vec{t}_n\cdot \nabla u_{j,\varepsilon_n}$ and integrating in $B(y_{j,\varepsilon_n}^{i_0},\delta_0\varepsilon_n^{-1})$, we get the following local Pohozaev type identity
\begin{equation}\label{eq4.27}
\begin{split}
&\frac{\varepsilon_n}{2}\int_{B(y_{j,\varepsilon_n}^{i_0},\delta_0\varepsilon_n^{-1})}u_{j,\varepsilon_n}^2(\nabla V(\varepsilon_n x)\cdot\vec{t}_n)dx+
\frac{1}{2}\int_{B(y_{j,\varepsilon_n}^{i_0},\delta_0\varepsilon_n^{-1})}\xi_nu_{j,\varepsilon_n}^2(\nabla\chi_{\varepsilon_n}\cdot\vec{t}_n)dx\\
&=\varepsilon_n\int_{B(y_{j,\varepsilon_n}^{i_0},\delta_0\varepsilon_n^{-1})}F(u_{j,\varepsilon_n})(\nabla K(\varepsilon_n x)\cdot\vec{t}_n)dx+\frac{1}{2}\int_{\partial B(y_{j,\varepsilon_n}^{i_0},\delta_0\varepsilon_n^{-1})}\phi u_{j,\varepsilon_n}^2(\vec{t}_n\cdot \nu)ds\\
&\quad-\frac{1}{2}\int_{B(y_{j,\varepsilon_n}^{i_0},\delta_0\varepsilon_n^{-1})}u_{j,\varepsilon_n}^2(\nabla \phi\cdot\vec{t}_n)dx-
\int_{\partial B(y_{j,\varepsilon_n}^{i_0},\delta_0\varepsilon_n^{-1})}K(\varepsilon_n x)F(u_{j,\varepsilon_n})(\vec{t}_n\cdot \nu)ds\\
&\quad+\frac{1}{2}\int_{\partial B(y_{j,\varepsilon_n}^{i_0},\delta_0\varepsilon_n^{-1})}|\nabla u_{j,\varepsilon_n}|^2(\vec{t}_n\cdot \nu)ds-
\int_{\partial B(y_{j,\varepsilon_n}^{i_0},\delta_0\varepsilon_n^{-1})}(\nabla u_{j,\varepsilon_n} \cdot \vec{t}_n)(\nabla u_{j,\varepsilon_n}\cdot \nu)ds\\
&\quad+\frac{1}{2}\int_{\partial B(y_{j,\varepsilon_n}^{i_0},\delta_0\varepsilon_n^{-1})}V(\varepsilon_n x)u_{j,\varepsilon_n}^2(\vec{t}_n\cdot \nu)ds
+\frac{1}{2}\int_{\partial B(y_{j,\varepsilon_n}^{i_0},\delta_0\varepsilon_n^{-1})}\xi_n \chi_{\varepsilon_n}(x)u_{j,\varepsilon_n}^2(\vec{t}_n\cdot \nu)ds,
\end{split}
\end{equation}
where $\nu$ denotes the unit outward normal to the boundary of $B(y_{j,\varepsilon_n}^{i_0},\delta_0\varepsilon_n^{-1})$.
\par
By \eqref{eq4.23} and $u_{j,\varepsilon_n}(\cdot+y_{y,\varepsilon_n}^{i_0})\rightharpoonup U_j^{i_0}$ in $H^1(\mathbb{R}^3)$, we get that, for sufficiently large $n$,
\begin{equation*}
\begin{split}
&\quad\varepsilon_n\int_{B(y_{j,\varepsilon_n}^{i_0},\delta_0\varepsilon_n^{-1})}(\vec{t}_n\cdot (\nabla V)(\varepsilon_n x))u_{j,\varepsilon_n}^2dx\\
&\geq\frac{\varepsilon_n}{2}|\nabla V(y_j^{i_0})|^2\int_{B(0,\delta_0\varepsilon_n^{-1})}u_{j,\varepsilon_n}^2(\cdot+ y_{j,\varepsilon_n}^{i_0})dx\\
&\geq C\varepsilon_n
\end{split}
\end{equation*}
where
\begin{equation*}
C=\frac{1}{4}|\nabla V(y_j^{i_0})|^2\int_{\mathbb{R}^3}|U_j^{i_0}|^2dx>0.
\end{equation*}
From \eqref{eq4.22}, we get that, for any $x\in B(y_{j,\varepsilon_n}^{i_0},\delta_0\varepsilon_n^{-1})\setminus \Lambda_{\varepsilon_n}$,
\begin{equation*}
\nabla V(\varepsilon_n y_{j,\varepsilon_n}^{i_0})\cdot \nabla \text{dist}(x,\partial \Lambda_{\varepsilon_n})>0.
\end{equation*}
It follows that, for any $x\in B(y_{j,\varepsilon_n}^{i_0},\delta_0\varepsilon_n^{-1})$,
\begin{equation}\label{eq4.28}
\vec{t}_n\cdot \nabla \chi_{\varepsilon_n}(x)\geq 0.
\end{equation}
Hence the left hand side of \eqref{eq4.27}
\begin{equation}\label{eq4.29}
\text{LHS}\geq C\varepsilon_n.
\end{equation}
On the other hand, by \eqref{eq4.24} and the definition of $F$, we deduce
\begin{equation}\label{eq4.30}
\varepsilon_n\int_{B(y_{j,\varepsilon_n}^{i_0},\delta_0\varepsilon_n^{-1})}F(u_{j,\varepsilon_n})(\nabla K(\varepsilon_n x)\cdot\vec{t}_n)dx\leq 0.
\end{equation}
Note that
\begin{equation}\label{eq4.31}
\begin{split}
\phi(x)=\phi_{u_{j,\varepsilon_n}}(x)&=\int_{\mathbb{R}^3}\frac{u_{j,\varepsilon_n}^2(y)}{|x-y|}dy=\int_{|x-y|\leq 1}\frac{u_{j,\varepsilon_n}^2(y)}{|x-y|}dy
+\int_{|x-y|\geq 1}\frac{u_{j,\varepsilon_n}^2(y)}{|x-y|}dy\\
&\leq \int_{|x-y|\leq 1}\frac{u_{j,\varepsilon_n}^2(y)}{|x-y|}dy
+\int_{|x-y|\geq 1}u_{j,\varepsilon_n}^2(y)dy\\
&\leq \bigg(\int_{|x-y|\leq 1}\frac{1}{|x-y|^{t^\prime}}dy\bigg)^{\frac{1}{t^\prime}}  \bigg(\int_{|x-y|\leq 1}u_{j,\varepsilon_n}^{2t}(y)dy\bigg)^{\frac{1}{t}}+\eta_k\\
&\leq C_k
\end{split}
\end{equation}
where $t^\prime<3,t\in [1,3],\frac{1}{t}+\frac{1}{t^\prime}=1$. Moreover, using H\"older inequality and the boundedness of $u_{j,\varepsilon_n}$ in $H^1(\mathbb{R}^3)$, one has
\begin{equation}\label{eq4.32}
\begin{split}
\bigg|\int_{B(y_{j,\varepsilon_n}^{i_0},\delta_0\varepsilon_n^{-1})}u_{j,\varepsilon_n}^2(\nabla \phi\cdot\vec{t}_n)dx\bigg|&\leq C\big(\int_{B(y_{j,\varepsilon_n}^{i_0},\delta_0\varepsilon_n^{-1})}u_{j,\varepsilon_n}^4dx\big)^{\frac{1}{2}}
\big(\int_{\mathbb{R}^3}|\nabla \phi |^2dx\big)^{\frac{1}{2}}\\
&\leq C\big(\int_{B(y_{j,\varepsilon_n}^{i_0},\delta_0\varepsilon_n^{-1})}u_{j,\varepsilon_n}^4dx\big)^{\frac{1}{2}}.
\end{split}
\end{equation}
Therefore, it follows from Lemma \ref{Lem4.6}, \eqref{eq4.25}-\eqref{eq4.26} and \eqref{eq4.30}-\eqref{eq4.32} that there exist $C > 0$ and $c > 0$ independent of $n$ such that, for
sufficiently large $n$, the right hand side of \eqref{eq4.27}
\begin{equation}\label{eq4.33}
\text{RHS}\leq C(\varepsilon_n^{-2}+\varepsilon_n^{-\frac{3}{2}})\text{exp}(-c\varepsilon_n^{-1}).
\end{equation}
From \eqref{eq4.29} and \eqref{eq4.33}, we get a contradiction for sufficiently large $n$. This completes this proof.
\end{proof}

\begin{Lem}\label{Lem4.8}
For any $0<\delta<\delta_0$, there exist $c=c(\delta,k)>$ and $C=C(\delta,k)>0$ independent of $\varepsilon$ such that for every $1\leq j \leq k$
\begin{equation*}
|u_{j,\varepsilon}(x)|\leq C\text{exp}(-c\text{dist}(x,\big(U(\delta)\big)_\varepsilon)),~x\in \mathbb{R}^3.
\end{equation*}
\end{Lem}
\begin{proof}
From \eqref{eq4.13} and the property $(4)$ of the profile decomposition \eqref{eq4.9}, there exists $R_0>0$ independent of $\varepsilon$ such that
\begin{equation*}
C(\frac{a_1}{2})|u_{j,\varepsilon}(x)|^{p-2}<\frac{a_1}{4},~\text{if}~x\in \mathbb{R}^3\setminus B(y_{j,\varepsilon}^i,R_0).
\end{equation*}
Note that, by Lemma \ref{Lem4.7}, for sufficiently small $\varepsilon>0$, there holds
\begin{equation*}
\{x\in \mathbb{R}^3|\text{dist}(x,\overline{\big(U(\delta)\big)_\varepsilon})\geq R_0\}\subset \mathbb{R}^3\setminus B(y_{j,\varepsilon}^i,R_0)
\end{equation*}
and thus
\begin{equation}\label{eq4.34}
C(\frac{a_1}{2})|u_{j,\varepsilon}(x)|^{p-2}<\frac{a_1}{4},~\text{if~dist}(x,\overline{\big(U(\delta)\big)_\varepsilon})\geq R_0.
\end{equation}
To prove the Lemma \ref{Lem4.8}, it suffices to prove
\begin{equation}\label{eq4.35}
|u_{j,\varepsilon}(x)|\leq C\text{exp}(-c\text{dist}(x,\big(U(\delta)\big)_\varepsilon)),~\text{if~dist}(x,\overline{\big(U(\delta)\big)_\varepsilon})\geq R_0.
\end{equation}
For $m\in \mathbb{N}$, let
\begin{equation*}
B_m=\{x\in \mathbb{R}^3|\text{dist}(x,\overline{\big(U(\delta)\big)_\varepsilon})\geq R_0+m-1\},
\end{equation*}
and let $\rho_m$ be a cut-off function satisfying that $0\leq \rho_m(t)\leq 1, |\rho_m^\prime(t)|\leq 4$ for all $t\in \mathbb{R}$ and
\begin{equation*}
\rho_m(t)=
\begin{cases}
0,&\text{if}~t\leq R_0+m-1,\\
1,&\text{if}~t\geq R_0+m.
\end{cases}
\end{equation*}
For $x\in \mathbb{R}^3$, let $\eta_m(x)=\rho_m(\text{dist}(x,\overline{\big(U(\delta)\big)_\varepsilon}))$. Multiplying both sides of \eqref{eq2.5} by $\eta_m^2u_{j,\varepsilon_n}$ and integrating on $\mathbb{R}^3$, we get that
\begin{equation}\label{eq4.36}
\begin{split}
&\int_{B_{m-1}}|\nabla u_{j,\varepsilon}|^2\eta_m^2dx +\int_{B_{m-1}}\phi_{u_{j,\varepsilon}}u_{j,\varepsilon}^2\eta_m^2dx+\int_{B_{m-1}}V(\varepsilon x)u_{j,\varepsilon}^2\eta_m^2dx\\
&\quad+\xi_\varepsilon\int_{B_{m-1}}\chi_{\varepsilon}(x)u_{j,\varepsilon}^2\eta_m^2dx-\int_{B_{m-1}}f(u_{j,\varepsilon})\eta_m^2u_{j,\varepsilon}dx\\
&=-2\int_{B_{m-1}}u_{j,\varepsilon}\psi_m\nabla u_{j,\varepsilon}\nabla \eta_mdx\\
&\leq 8\int_{B_{m-1}\setminus B_m}|u_{j,\varepsilon}|\cdot |\nabla u_{j,\varepsilon}|dx,
\end{split}
\end{equation}
where
\begin{equation*}
\xi_\varepsilon:=2\beta\big(\int_{\mathbb{R}^3}\chi_\varepsilon(x)u_{j,\varepsilon}^2dx-1\big)_+^{\beta-1}.
\end{equation*}
By \eqref{eq4.34}, we get that
\begin{equation}\label{eq4.37}
\begin{split}
&\int_{R_{m-1}}|\nabla u_{j,\varepsilon}|^2\eta_m^2dx +\int_{R_{m-1}}\phi_{u_{j,\varepsilon}}u_{j,\varepsilon_n}^2\eta_m^2dx+\int_{R_{m-1}}V(\varepsilon x)u_{j,\varepsilon}^2\eta_m^2dx\\
&\quad+\xi_\varepsilon\int_{R_{m-1}}\chi_{\varepsilon}(x)u_{j,\varepsilon_n}^2\eta_m^2dx-\int_{R_{m-1}}f(u_{j,\varepsilon})\eta_m^2u_{j,\varepsilon}dx\\
&\geq \min\{1,\frac{a_1}{4}\}\int_{R_m}(|\nabla u_{j,\varepsilon}|^2+u_{j,\varepsilon}^2)dx.
\end{split}
\end{equation}
Combining \eqref{eq4.36} and \eqref{eq4.37} yields that
\begin{equation*}
\begin{split}
\int_{B_m}(|\nabla u_{j,\varepsilon}|^2+u_{j,\varepsilon}^2)dx&\leq \frac{8}{\min\{1,\frac{a_1}{4}\}}\int_{R_m\setminus R_{m+1}}|u_{j,\varepsilon}|\cdot |\nabla u_{j,\varepsilon}|dx\\
&\leq \frac{4}{\min\{1,\frac{a_1}{4}\}}\int_{R_m\setminus B_{m+1}}(|\nabla u_{j,\varepsilon}|^2+u_{j,\varepsilon}^2)dx.
\end{split}
\end{equation*}
The rest of proof for \eqref{eq4.35} is similar to the proof of Lemma \ref{Lem4.6}.
\end{proof}

\begin{Lem}\label{Lem4.9}
There exists $\varepsilon_k>0$ such that $0<\varepsilon<\varepsilon_k$, then for every $1\leq j \leq k,u_{j,\varepsilon}$ is a solution of system \eqref{eq2.1}.
\end{Lem}
\begin{proof}
For any $0<\delta< \delta_0$, we have
\begin{equation*}
\text{dist}(U(\delta),\partial \Lambda)=\delta>0.
\end{equation*}
Furthermore, for any $x\in \mathbb{R}^3\setminus \Lambda_\varepsilon$, there exists $y_x\in (\overline{U(\delta)})_\varepsilon$ such that
\begin{equation*}
|x-y_x|=\text{dist}(x,(\overline{U(\delta)})_\varepsilon)=\text{dist}(x,(U(\delta))_\varepsilon)\geq \delta\varepsilon^{-1}.
\end{equation*}
Therefore, by Lemma \ref{Lem4.8}, we get that, for every $1\leq j \leq k$
\begin{equation*}
\begin{split}
\int_{\mathbb{R}^3}\chi_\varepsilon(x) u_{j,\varepsilon}^2dx&\leq\varepsilon^{-6}\int_{\mathbb{R}^3\setminus \Lambda_\varepsilon}u_{j,\varepsilon}^2dx\\
&\leq C\varepsilon^{-6}\int_{\mathbb{R}^3\setminus \Lambda_\varepsilon}e^{-2c\text{dist}(x,(U(\delta))_\varepsilon)}dx\\
&\leq C\varepsilon^{-6}\int_{\{x:|x-y_x|\geq \delta\varepsilon^{-1}\}}e^{-2c|x-y_x|}dx\\
&=C\varepsilon^{-6}\int_{\{x:|x|\geq \delta\varepsilon^{-1}\}}e^{-2c|x|}dx\\
&\leq C\varepsilon^{-8}e^{-2c \delta\varepsilon^{-1}}\rightarrow 0~\text{as}~\varepsilon\rightarrow 0.
\end{split}
\end{equation*}
Thus, there exists $\varepsilon_k>0$ such that for any $\varepsilon\in (0,\varepsilon_k)$
\begin{equation*}
\int_{\mathbb{R}^3}\chi_\varepsilon(x) u_{j,\varepsilon}^2dx\leq 1
\end{equation*}
which implies that $Q_\varepsilon(u_{j,\varepsilon})=0$ for any $\varepsilon\in (0,\varepsilon_k)$, and so $u_{j,\varepsilon}$ is a solution of system \eqref{eq2.1}.
\end{proof}

{\bf Proof of Theorem 1.1} The results of Theorem \ref{Thm1.1} follows from Theorem \ref{Thm3.2}, Lemma \ref{Lem4.8} and Lemma \ref{Lem4.9} immediately.
\par
\vspace{3mm}
{\bf Proof of Theorem 1.3} The proof of Theorem \ref{Thm1.3} is almost the same as that for Theorem \ref{Thm1.1} but easier, and the main difference is in Lemma \ref{Lem4.7} and Lemma \ref{Lem4.8}.
\begin{Lem}\label{Lem4.10}
For every $1\leq i\leq m_j , \lim\limits_{\varepsilon \rightarrow 0}\text{dist}(\varepsilon y_{j,\varepsilon}^i,\mathcal{A}) = 0$.
\end{Lem}
\begin{proof}
Arguing indirectly, we assume that there exist $1\leq i_0\leq m_j$ and sequence $\varepsilon_n>0$ such that $\lim\limits_{n\to \infty}\varepsilon_n = 0$ and
 \begin{equation*}
\lim_{n\rightarrow\infty}\text{dist}(\varepsilon_n y_{j,\varepsilon_n}^{i_0},\mathcal{A})>0.
\end{equation*}
It follows from $y_j^{i_0}= \lim\limits_{n\to \infty}\varepsilon_n y_{j,\varepsilon_n}^{i_0}\notin \mathcal{A}$ that $\nabla V(y_j^{i_0})\neq 0$, and then \eqref{eq4.23} is true. For other details in the proof, we leave them to the readers and omit it here.
\end{proof}

Using Lemma \ref{Lem4.10} and a similar argument as in the proof of Lemma \ref{Lem4.8}, we can obtain the following Lemma from line to line.
\begin{Lem}\label{Lem4.11}
For any $\delta>0$, there exist $c=c(\delta,k)>$ and $C=C(\delta,k)>0$ independent of $\varepsilon$ such that for every $1\leq j \leq k$
\begin{equation*}
|u_{j,\varepsilon}(x)|\leq C\text{exp}(-c\text{dist}(x,\big(\mathcal{A}^\delta\big)_\varepsilon)),~x\in \mathbb{R}^3.
\end{equation*}
\end{Lem}

\section{Final remark} 1.~We note that the all conclusions of Theorem \ref{Thm1.1}-Theorem \ref{Thm1.4} are still hold if $f$ is only continuous. In fact, the operator $A_\varepsilon$ which defined in section 3 is not applicable to construct a descending flow for $\Phi_\varepsilon$ since $A_\varepsilon$ may be only continuous. Fortunately, using the same argument in \cite[Lemma 4.1]{Bartsch-Liu2004JDE} and \cite[Lemma 2.1]{Bartsch-Liu-Weth2005PLMS}, there exists a locally Lipschitz continuous operator $B_\varepsilon$ which inherits the main properties of $A_\varepsilon$. More precisely, $B_\varepsilon$ is defined on $X_0$ and there exists $\sigma_0>0$ such that
\begin{itemize}
\item[$(i)$]  $B_\varepsilon(\partial P_\sigma^+)\subset P_\sigma^+$ and $B_\varepsilon(\partial P_\sigma^-)\subset P_\sigma^-$ for $\sigma\in (0,\sigma_0)$;
\item[$(ii)$]  $\frac{1}{2}\|u-B_\varepsilon(u)\|\leq \|u-A_\varepsilon(u)\|\leq 2\|u-B_\varepsilon(u)\|$ for all $u\in X_0$;
\item[$(iii)$]  $\langle \Phi_\varepsilon^\prime(u),u-B_\varepsilon(u)\rangle\geq \frac{1}{2}\|u-A_\varepsilon(u)\|^2$ for all $u\in X_0$;
\item[$(iv)$]  $B_\varepsilon$ is odd;
\end{itemize}
where $X_0=H^1(\mathbb{R}^3)\setminus K_\varepsilon$ and $K_\varepsilon$ denotes the set of fixed points of $A_\varepsilon$, which is exactly the set of critical points of $\Phi_\varepsilon$. Consequently, the operator $B_\varepsilon$ satisfies assumptions $(A_1)$ and $(A_2)$. Therefore, we always assume that $A_\varepsilon$ is locally Lipschitz continuous on $X_0$ and then all conclusions are hold.
\par
\vspace{3mm}
2.~Our method is works for Schr\"odinger equation. Precisely, we consider the following nonlinear Schr\"odinger equation
\begin{equation*}
-\varepsilon^2 \Delta u+V(x)u=K(x)f(u),~\text{in}~\mathbb{R}^N
\end{equation*}
where $\varepsilon>0$ is a small parameter and $N\geq 2$. Under the assumptions $(f_1)$-$(f_4)$, $(V)$, $(K)$, $(VK1)$ and $(VK2)$, all conclusions of Theorem \ref{Thm1.1}-Theorem \ref{Thm1.4} are still hold. Thus, the result generalizes the result in \cite{Chen-Wang2017CVPDE}.
\par
\vspace{3mm}
3.~As mentioned in \cite{Chen-Wang2017CVPDE}, the conclusions of Theorem \ref{Thm1.3} seems to hold for nonlinear Schr\"odinger-Poisson system with a critical frequency, i.e., we consider the following system
\begin{equation}\label{eq5.1}
\begin{cases}
-\varepsilon^2 \Delta u+V(x)u+\phi u=f(u),&\text{in}~\mathbb{R}^3,\\
-\varepsilon^2 \Delta \phi=u^2,&\text{in}~\mathbb{R}^3,
\end{cases}
\end{equation}
where $f$ satisfies $(f_1)$-$(f_4)$ above and $V$ satisfies
\begin{itemize}
\item[$(V_1)$]  $0=\inf\limits_{x\in \mathbb{R}^3}V(x)<\liminf\limits_{|x|\rightarrow\infty}V(x)$.
\item[$(V_2)$]  There exists a closed subset $\mathcal{Z}$ with a nonempty interior such that $V(x)=0$ for $x\in \mathcal{Z}$.
\end{itemize}
We can prove the following Theorem, and we leave the details to the interested readers.
\begin{Thm}\label{Thm5.1}
Assume that $(f_1)$-$(f_4)$ and $(V_1)$-$(V_2)$ hold. Then for any positive integer $k$, there exists $\varepsilon_k > 0$ such that if $0 <\varepsilon <\varepsilon_k$, system \eqref{eq5.1} has at least $k$ pairs of sign-changing solutions $\pm v_{j,\varepsilon}, j = 1, 2, \cdots, k$. Moreover, for any $\delta> 0$, there exist $c = c(\delta, k) > 0$ and $C = C(\delta, k) > 0$ such that
\begin{equation*}
|v_{j,\varepsilon}(x)|\leq C\text{exp}\bigg(-\frac{c\text{dist}(x,\mathcal{Z}^\delta)}{\varepsilon}\bigg)~\text{for}~x\in \mathbb{R}^3,~j=1,\cdots,k.
\end{equation*}
\end{Thm}

\section*{Acknowledgments}
We should like to thank the anonymous referee for his/her careful readings of our manuscript and the useful comments
made for its improvement. The work was supported by the National Science Foundation of China (NSFC11871242).

\bibliographystyle{plain}
\bibliography{System}

\end{document}